\documentclass[11pt,a4paper]{article}
\usepackage[english]{babel}
\usepackage[utf8]{inputenc}
\usepackage[T1]{fontenc}
\usepackage{graphicx}
\usepackage{amsmath,amssymb,amsfonts,amsthm}
\usepackage{geometry}
\usepackage{mathrsfs}
\usepackage{nicefrac}
\usepackage{booktabs}
\usepackage{multirow}
\usepackage{enumitem}
\usepackage{float}
\usepackage{hyperref}
\usepackage{blkarray}
\usepackage{upgreek}
\usepackage{xcolor}
\usepackage[alphabetic]{amsrefs}
\usepackage[symbol]{footmisc}
\hypersetup{colorlinks = true, linkcolor=red, citecolor=blue}

\usepackage[ruled]{algorithm2e}
\usepackage{authblk}
\usepackage{subcaption}
\usepackage{bm} 

\newcommand{\psh}[2]{\langle #1 \, , #2 \rangle}

\newcommand{\sumlim}[2]{\sum\limits_{#1}^{#2}}

\newcommand{\R}{\mathbb R}
\newcommand{\N}{\mathbb N}

\newcommand{\C}{\mathbb C}

\DeclareMathOperator*{\argmin}{arg\,min}

\begin{document}
\numberwithin{equation}{section}
\newtheorem{theo}{Theorem}[section]
\newtheorem{prop}[theo]{Proposition}
\newtheorem{note}[theo]{Remark}
\newtheorem{lem}[theo]{Lemma}
\newtheorem{cor}[theo]{Corollary}
\newtheorem{definition}[theo]{Definition}
\newtheorem{assumption}{Assumption}

\theoremstyle{definition}
\newtheorem{example}{Example}[section]

\title{Inversion symmetry of singular values and a new orbital ordering method  in tensor train approximations for \\ quantum chemistry}
\author[1]{Mi-Song Dupuy}
\author[1]{Gero Friesecke}
\affil{Faculty of Mathematics, TU Munich, Germany, {\tt dupuy@ma.tum.de}, {\tt gf@ma.tum.de}}
\date{February 17, 2020}

\renewcommand\Affilfont{\itshape\small}

\maketitle
\begin{abstract}
\begin{small}
The tensor train approximation of electronic wave functions lies  at the core of the QC-DMRG (Quantum Chemistry Density Matrix Renormalization Group) method, a recent state-of-the-art method for numerically solving the $N$-electron Schr\"odinger equation. 
It is well known that the accuracy of TT approximations is governed by the tail of the associated singular values, which in turn strongly depends on the ordering of the one-body basis.

Here we find that the singular values $s_1\ge s_2\ge ... \ge s_d$ of tensors representing ground states of noninteracting Hamiltonians possess a surprising inversion symmetry, $s_1s_d=s_2s_{d-1}${}$=s_3s_{d-2}=...$, thus reducing the tail behaviour to a single hidden invariant, which moreover depends explicitly on the ordering of the basis. For correlated wavefunctions, we find that the tail is upper bounded by a suitable superposition of the invariants. Optimizing the invariants or their superposition thus provides a new ordering scheme for QC-DMRG. 
Numerical tests on simple examples, i.e. linear combinations of a few Slater determinants, show that the new scheme reduces the tail of the singular values by several orders of magnitudes over existing methods, including the widely used Fiedler order.
\end{small}
\end{abstract}

%

Solving the $N$-body electronic Schr\"odinger equation is a formidable numerical challenge which has been recently tackled by tensor methods inspired by schemes used in spin chain theory. 
In the numerical treatment of one-dimensional spin chain systems, the Density Matrix Renormalization Group (DMRG) introduced in \cite{white1992dmrg,white1993dmrg} has become the state-of-the-art method.
As was found out later \cite{ostlund1995thermodynamic,dukelsky1998equivalence}, the eigenfunctions in the DMRG method are tensor trains (TT), also known as matrix product states (MPS) in physics.

By writing the electronic Schr\"odinger equation in the Fock space of the one-body basis functions and taking an occupation number viewpoint, the electronic wave function can be seen as a tensor where the one-body basis functions play the role of spin sites. 
In this setting, it becomes natural to transfer ideas of the DMRG method. The resultant method is known as Quantum Chemistry-DMRG (QC-DMRG)   \cite{white1999abinitio,mitrushenkov2001quantum,chan2002highly,legeza2003controlling} and has recently been implemented in \cite{Wouters2014CheMPS2,chan2015block-dmrg,keller2015qcmaquis,legeza2018qc-dmrg-budapest}. For a general introduction to the method see \cite{schneider2014tensor}. 
The ground state electronic wave function in QC-DMRG is given by a TT, for which it is well-known that the quality of its approximation is governed by the tail behaviour of the singular values of the matricizations of the original tensor \cite{hackbusch2012tensor}. 
While an interesting mathematical literature on the singular values of tensors is emerging \cite{Hackbusch2017interconnection,shi2018numericalranks,griebel2019analysis}, 
theoretical understanding of the following fundamental questions is still lacking: \\
-- {\it How does the tail of the singular values behave for ground and excited states of realistic quantum} \\
\textcolor{white}{--} {\it chemical systems?} \\
-- {\it How to choose a good network to approximate the states of interest?} \\
The latter is a central issue not just in QC-DMRG but in tensor approximations generally. When the network is a tensor train associated with a given one-particle basis, it reduces to: {\it how to choose the ordering of the basis?} 

The goal of this paper is to study these questions in detail for a simplified but important class of quantum states, namely ground states of non-interacting Hamiltonians, which are given by Slater determinants. 
We establish a surprising inversion symmetry of the singular values which reveals that the tail behaviour is determined by a single hidden invariant. The latter depends explicitly on the ordering of the background basis.
By optimizing the invariant with respect to the orbital ordering, improvements  by four to five orders of magnitude of the decay of the singular values are typically  achieved; see Figure \ref{fig:slater-comparison} in section \ref{sec:testslat}. We term this method {\it best prefactor ordering}. By contrast, the widely used Fiedler order, which is based on an entanglement analysis of the basis \cite{barcza2011quantum}, only gives an improvement by one order of magnitude. 

We then show that the tail size for general superpositions of Slater determinants is upper-bounded by a suitable superposition of the invariants, and propose a corresponding ordering scheme in the general case.
Our ordering scheme is tested on linear combinations of a few Slater determinants which capture important features of typical ground and excited state wave functions in quantum chemistry. 
Our numerical results in section \ref{sec:testcorr} show that it outperforms all existing methods, including the Fiedler order.

\section{The tensor-train decomposition in quantum chemistry}
The goal of this section is to recall how TT approximations in quantum chemistry are set up and to introduce notation.

\subsection{The electronic Schr\"odinger equation}

A quantum mechanical system of $N$ non-relativistic electrons is completely described by a wave function $\Psi$ depending on $3N$ spatial variables $\mathbf{r}_i \in \R$, $i=1,\dots,N$ and discrete spin variables $s_i \in \lbrace \pm \frac{1}{2} \rbrace$ , $i=1,\dots,N$
\begin{equation}
    \Psi \ : 
    \begin{cases}
    (\R^{3} \times \lbrace \pm \frac{1}{2} \rbrace)^N & \to \mathbb{C} \\
    (\mathbf{r}_1,s_1; \dots; \mathbf{r}_N,s_N) & \mapsto \Psi(\mathbf{r}_1,s_1; \dots;\mathbf{r}_N,s_N).
    \end{cases}
\end{equation}
The Pauli exclusion principle states that the wave function of electrons must be antisymmetric with respect to permutations of variables,
\begin{equation} \label{anti}
    \Psi(\dots; \mathbf{r}_i, {s}_i; \dots;\mathbf{r}_j, {s}_j; \dots) = - \Psi(\dots; \mathbf{r}_j, {s}_j; \dots;\mathbf{r}_i, {s}_i; \dots).
\end{equation}
The wave function $\Psi$ belongs to the Hilbert space $\bigwedge_{i=1}^N L^2(\R^3 \times \lbrace \pm \frac{1}{2} \rbrace) = \{\Psi \in L^2((\R^3\times\lbrace\pm\frac12\rbrace)^N) \, : \, \eqref{anti} \}$.

The central goal of quantum chemistry is to numerically solve the electronic Schr\"odinger equation 
\begin{equation}
 H \Psi = E \Psi.   
\end{equation}
Here $H$ is a partial differential operator of the form
\begin{equation}
    H = \sumlim{i=1}{N} \left( -\frac{1}{2} \Delta_{\mathbf{r}_i} + v(\mathbf{r}_i) \right)
       + \sumlim{1 \leq i <j \leq N}{} v_{ee}(\mathbf{r}_i-\mathbf{r}_j)
\end{equation}
and $v\, : \, \R^3\to\R$, $v_{ee}\, : \, \R^3\to\R$ are suitable potentials. Explicitly, for a molecule with $M$ atomic nuclei at positions $\mathbf{R_I}\in\R^3$ and with charges $Z_I>0$ ($I=1,...,M$), 
$v(\mathbf{r})= - \sum_{I=1}^{M} {Z_I}/|\mathbf{R}_I - \mathbf{r}_i|$ (electron-nuclei interaction) and $v_{ee}(\mathbf{r}_i-\mathbf{r}_j)=1/|\mathbf{r}_i-\mathbf{r}_j|$ (electron-electron interaction). 

By Zhislin's theorem (see \cite{friesecke2003mcscf} for a short proof), the operator $H$ acting on $\bigwedge_{i=1}^N L^2(\R^3 \times \lbrace \pm \frac{1}{2} \rbrace)$ with domain $\bigwedge_{i=1}^N H^2(\R^3 \times \lbrace \pm \frac{1}{2} \rbrace)$ has countably many discrete eigenvalues below its essential spectrum if $\sum_{I=1}^{M} Z_I > N-1$ (in particular when the system is neutral). 

It is of particular interest to calculate the lowest eigenvalue and eigenstate of $H$, called ground-state energy respectively ground state of the system, which satisfy the Rayleigh-Ritz variational principle 
\begin{align}
    E_0 &= \min \lbrace \psh{\Psi}{H \Psi} \ : \ \psh{\Psi}{\Psi} = 1, \ \Psi \in \mathcal{V}_N \rbrace, \label{RR1} \\
    \Psi_0 &= \argmin \lbrace \psh{\Psi}{H \Psi} \ : \ \psh{\Psi}{\Psi} = 1, \ \Psi \in \mathcal{V}_N 
    \rbrace, \label{RR2}
\end{align}
where $\langle \cdot , \cdot\rangle$ is the inner product on $L^2((\R^3\times\{\pm\frac12\})^N)$ and $\mathcal{V}_N$ is the variational space
$
    \mathcal{V}_N = \bigwedge_{i=1}^N H^1(\R^3 \times \lbrace \pm \tfrac{1}{2} \rbrace).
$
%
%
\subsection{Full configuration interaction} \label{sec:FCI}

Starting point of most computational methods for \eqref{RR1}--\eqref{RR2} is the following ``folklore theorem'' which could be made rigorous e.g. for the molecular potential $v$ above.
\begin{prop}  
\label{prop:slater_ground_state_noninteracting_hamiltonian}
If $H=\sum_{i=1}^N(-\frac12 \Delta_{\mathbf{r}_i} + v(\mathbf{r}_i))$, i.e. when the electron-electron potential $v_{ee}$ is zero, there exists a solution $\Psi_0$ to \eqref{RR2} which has the form of a Slater determinant $\Psi_0=|\psi_1,...,\psi_N\rangle$ (see \eqref{Slater} below) for some functions $\psi_1,...,\psi_N\in H^1(\R^3\times\{\pm\frac12\} )$. 
\end{prop}
Explicitly, the $\psi_i$ ($i=1,...,N$) could be taken as the lowest $N$ eigenstates of the one-body operator $h=-\frac12 \Delta_{\mathbf{r}} + v(\mathbf{r})$ acting on $L^2(\R^2\times\{\pm\frac12\} )$ with domain $H^2(\R^3\times \{\pm\frac12\} )$. Slater determinants are special $N$-electron wavefunctions in $\mathcal{V}_N$ which have the form
\begin{equation} \label{Slater}
      \Psi (r_1,s_1,...,r_N,s_N) = |\psi_1,...,\psi_N\rangle (r_1,s_1,...,r_N,s_N) = 
      \frac{1}{\sqrt{N!}} \det (\psi_i(\mathbf{r}_j,s_j))_{i,j=1}^N
\end{equation}
for some functions $\psi_1,...,\psi_N\in H^1(\R^3\times\{\pm \frac12\} )$ which are orthonormal in $L^2(\R^3\times\{\pm\frac12\})$. 

When $v_{ee}$ is nonzero, this result no longer holds, but minimizing the functional in \eqref{RR1}--\eqref{RR2} over Slater determinants often gives a reasonable first approximation to the ground state, the so-called Hartree-Fock determinant. 

General elements of $\mathcal{V}_N$ can be expanded in a basis of Slater determinants, leading to the method of full configuration interaction (FCI). In this method (see, e.g., \cite{helgaker2014molecular}), one starts from a countable set $\{\varphi_i\}_{i=1}^\infty$ of functions in $H^1(\R^3\times\{\pm\frac12\} )$ such that $\mathrm{Span}\{\varphi_i\}_{i=1}^\infty = H^1(\R^3 \times \lbrace \pm \frac{1}{2} \rbrace)$. One then truncates this set to a finite set 
$$
            \{\varphi_i\}_{1 \leq i \leq L}
$$
of these single-particle functions, known in quantum chemistry as a {\it single-particle basis set}, and denotes 
$
   \mathcal{V}_1^L =\mbox{Span}\{\varphi_1,...,\varphi_L\} \subset H^1(\R^3\times\{\pm\frac12\}).
$
In practice the first $N$ elements of this set are almost always taken as the single-particle functions $\psi_i$ which appear in a numerically computed Hartree-Fock determinant (these functions are called canonical orbitals), but the following applies to arbitrary single-particle basis sets.    

Let $1\leq i_1 < \dots < i_N \leq L$ be $N$ different indices, and denote the resulting Slater determinant by 
\begin{equation}
  |\varphi_{i_1},...,\varphi_{i_N}\rangle =: \Phi_{[i_1,\dots,i_N]}. 
\end{equation}
The full configuration interaction (full CI) space of an $N$-electron system associated with the  single-particle space $\mathcal{V}_1^L$ is the finite dimensional space spanned by all the above Slater determinants, i.e.
\[
\mathcal{V}^L_N = \mbox{Span}\lbrace \Phi_{[i_1,\dots,i_N]} \ : \ 1 \leq i_1 < \dots < i_N \leq L \rbrace \subset \mathcal{V}_N.
\]
The single particle functions $\varphi_i$ are orthonormal, hence the family of Slater determinants \linebreak $(\Phi_{[i_1,\dots,i_N]})_{1 \leq i_1 < \dots < i_N \leq N}$ also forms an orthonormal family. The dimension of the approximation space $\mathcal{V}^L_N$ is thus $\binom{L}{N}\sim L^N$.

The full CI approximation of \eqref{RR1}--\eqref{RR2} using the variational subspace $\mathcal{V}^L_N$ scales combinatorially with the number of electrons. It is thus only tractable for systems with a small number of electrons.

\subsection{Fock space and occupation representation} \label{sec:Fock}

The full CI space $\mathcal{V}^L_N$ is embedded into a larger space $\mathcal{F}^L$ called discrete Fock space,
\begin{equation}
    \mathcal{F}^L := \bigoplus\limits_{M=0}^L \mathcal{V}^L_M.
\end{equation}
The elements of $\mathcal{F}^L$ are of form $\Psi_0\oplus\Psi_1\oplus \cdots \oplus \Psi_L$, where $\Psi_M$ is an $M$-particle wavefunction belonging to $\mathcal{V}^L_M$. The dimension of $\mathcal{F}^L$ is 
$\dim \mathcal{F}^L = \sumlim{M=0}{L} \binom{L}{M} = 2^L$. 

In tensor-train approximations and QC-DMRG, an important role is played by the following  alternative representation of elements of the Fock space. This representation is particularly simple for the Slater determinants $\Phi_{[i_1,...,i_N]}$, where it just corresponds to a labelling by a binary string $(\mu_1,...,\mu_L)\in\{0,1\}^L$ indicating the presence of absence of the orbital $\varphi_i$. In order to allow linear combinations, i.e. quantum superpositions, we now associate, with any binary string $(\mu_1,...,\mu_L)\in\{0,1\}^L$, the element
\begin{equation}
   \Upphi_{(\mu_1,...,\mu_L)} = \chi_{\mu_1} \otimes \chi_{\mu_2} \otimes  \cdots \otimes \chi_{\mu_L} \in \bigotimes_{i=1}^L \C^2
\end{equation}
where $\{\chi_0,\chi_1\}$ is an orthonormal basis of $\C^2$. The state $\Upphi_{(\mu_1,...,\mu_L)}$ has the same information content as the binary string since the $i^{th}$ tensor factor $\chi_{\mu_i}$ just indicates whether or not the orbital $\varphi_i$ is occupied in the original Slater determiant $\Phi_{[i_1,...,i_N]}$: if $\chi_{\mu_i}=\chi_1$ (resp. $\chi_0$) then $\varphi_i$ is occupied (resp. unoccupied). 

With a general $N$-electron wavefunction in the full CI space $\mathcal{V}_N^L$, 
$\Psi = \sum\limits_{1\le i_1< ... < i_N\le L} c_{i_1...i_N} \Phi_{[i_1,...,i_N]}$, we now associate the element
\begin{equation} \label{occrep1}
    \Uppsi = \sum_{\mu_1=0}^1 \dots \sum_{\mu_L=0}^1 \Uppsi_{\mu_1,...,\mu_L} \Phi_{(\mu_1,...,\mu_L)} \in \bigotimes_{i=1}^L \C^2
\end{equation}
where
\begin{equation} \label{occrep2}
     \Uppsi_{\mu_1,...,\mu_L} = \begin{cases} 0 & \mbox{if }\sum_{i=1}^L\mu_i \neq N \\
                        c_{i_1...i_N} & \mbox{if }\mu_i=1 \mbox{ precisely when }i\in\{i_1,...,i_N\}, \, i_1<...<i_N. \end{cases}
\end{equation}
We call \eqref{occrep1}--\eqref{occrep2} the {\it occupation representation} of the wavefunction $\Psi$. By the normalization of $\Psi$, $\sum_{\mu_1=0}^{1} \dots \sum_{\mu_L=0}^{1} |\Uppsi_{\mu_1,\dots,\mu_L}|^2 = 1$. 
The fact that the coefficients $\Uppsi_{(\mu_1,\dots,\mu_L)}$ are zero when $\sum_{i=1}^L\mu_i\neq N$, which comes from the fact that $\Psi$ is an $N$-electron wavefunction, is crucial to bound the bond dimension of the tensor-train representation of $\Uppsi$ in Section~\ref{sec:proofs}.

We remark that numerical implementations of QC-DMRG typically use a slight modification of \eqref{occrep1}--\eqref{occrep2} in which the basis set is assumed to be of form $\{\varphi_1\!\!\uparrow,\varphi_1\!\!\downarrow,...,\varphi_{L/2}\!\uparrow,\varphi_{L/2}\!\downarrow\}$ with spatial orbitals $\varphi_i\in H^1(\R^3)$, where  $(\varphi_i)\!\uparrow)(r,s)=\varphi_i(r)\delta_{1/2}(s)$, $(\varphi_i\!\downarrow)(r,s)=\varphi_i(r)\delta_{-1/2}(s)$. One then works, instead of $\otimes_{\ell=1}^L \C^2$, in the space $\otimes_{\ell=1}^{L/2}\C^4$, with $\C^4$ being the span of the four basis functions $\chi_{-}$, $\chi_\uparrow$, $\chi_\downarrow$, $\chi_{\uparrow\downarrow}$. This basis represents the four occupation possibilities of the spatial orbital $\varphi_i$ (absent, present with upspin only, present with downspin only, present both with upspin and downspin). The results in this paper could easily be adapted to this setting.

Because the representation \eqref{occrep1}--\eqref{occrep2} of a quantum wavefunction is somewhat abstract, we give an example. The example will be revisited later to compare our new orbital ordering scheme to previous methods including the Fiedler order.
\begin{example}[Minimal basis H$_2$]
This is an example with $N=2$ and $L=4$. Consider an H$_2$ molecule with nuclei clamped at $R_A$ and $R_B$, and single-particle space $\mathcal{V}_1^L$ given by the Span of the two $1s$ orbitals $\chi_A(r)=e^{-|r-R_A|}/\sqrt{\pi}$, $\chi_B(r)=e^{-|r-R_B|}/\sqrt{\pi}$ of the individual H atoms multiplied by the spin functions $\delta_{\pm 1/2}(s)$. To obtain a canonical orthonormal basis of $\mathcal{V}_1^L$, let $\varphi_A$, $\varphi_B$ be the associated bonding respectively antibonding orbitals, 
$$
    \varphi_A(r)=\frac{\chi_A+\chi_B}{\sqrt{2+2S_{AB}}}, \;\;\;
    \varphi_B(r)=\frac{\chi_A-\chi_B}{\sqrt{2-2S_{AB}}}, \;\;\;
$$
where $S_{AB}$ is the overlap integral $\int_{\R^3} \chi_A(r)\chi_B(r)\, dr$. Our single-particle basis is then $\{\varphi_A\uparrow, \varphi_A\downarrow, \varphi_B\uparrow, \varphi_B\downarrow\}$. Consider now the Slater determinant
\begin{equation} \label{H2Slater}
    \Psi = \Big| (c\varphi_A + s \varphi_B)\uparrow, \, (c'\varphi_A + s'\varphi_B)\downarrow \Big\rangle
\end{equation}
for some coefficients $c$, $s$, $c'$, $s'\in\R$ with $c^2+s^2=c'{}^2+s'{}^2=1$. We note that the unrestricted Hartree-Fock (UHF) ground state of minimal-basis H$_2$ has the above form, for any bondlength $R=|R_A-R_B|$; moreover $(c,s)\neq (c',s')$ when $R$ is large \cite{szabo1982modern}. The occupation representation \eqref{occrep1}--\eqref{occrep2} with respect to our single-particle basis is, by expanding $\Psi$ and using $|\varphi_B\uparrow,\varphi_B\downarrow\rangle = - |\varphi_A\uparrow,\varphi_B\downarrow\rangle$,  
\begin{equation} \label{H2Slateroccrep}
     \Uppsi = cc'\Phi_{(1100)} + cs'\Phi_{(1001)} -sc'\Phi_{(0110)} + ss'\Phi_{(0011)} \in \bigotimes_{i=1}^4 \C^2. 
\end{equation}
\end{example}

In general, the tensor $\Uppsi=(\Uppsi_{\mu_1 ...\mu_L})_{\mu_1,...,\mu_L=0}^1$ belongs to the tensor product space $\bigotimes_{i=1}^L \mathbb{C}^2$, which has dimension $2^L$. The high dimension of this space urges one to investigate compact representations of these tensors; see section \ref{sec:TT}. 
We also emphasize that the occupation representation of an $N$-particle wavefunction in $\mathcal{V}_N^L$, and its entanglement, depends on the choice and ordering of the one-particle basis $\{\varphi_1,...,\varphi_L\}$ of $\mathcal{V}_1^L$. This issue is taken up in section \ref{sec:ordering}.

\subsection{Tensor-train decomposition} \label{sec:TT}

Tensor-trains (TT), also called matrix product states (MPS) in physics, are states $(\Uppsi_{\mu_1,...,\mu_L})$ of the following form:
\begin{align}
    \label{eq:tensor-train-decomposition}
    \forall \, (\mu_1,\dots,\mu_L) \in \{0,1\}^L, \ \Uppsi_{\mu_1,\dots,\mu_L} &= A_1[\mu_1] A_2[\mu_2] \dots A_L[\mu_L] \\
    &= \sum_{\alpha_1=1}^{r_1} \sum_{\alpha_2=1}^{r_2} ... \!\! \sum_{\alpha_{L-1}=1}^{r_{L-1}}\!\!
     (A_1[\mu_1])_{\alpha_1} (A_2[\mu_2])_{\alpha_1\alpha_2} \cdots 
     (A_L[\mu_L])_{\alpha_{L-1}}. \nonumber 
\end{align}
For each $(\mu_1,\dots,\mu_L) \in \{0,1\}^L$, $A_1[\mu_1]$ is a row vector of size $r_1$, for $k=2,\dots,L-1$, $A_k[\mu_k]$ is a $r_{k-1} \times r_k$ matrix, and $A[\mu_L]$ is a column vector of size $r_{L-1}$. See Figure \ref{F:TT} for a graphical representation. The number $r = \max\{r_k \, : \,  k=1,\dots,L-1\}$ is called the bond dimension.

\begin{figure}[http!]
\begin{center}
\includegraphics[width=0.75\textwidth ]{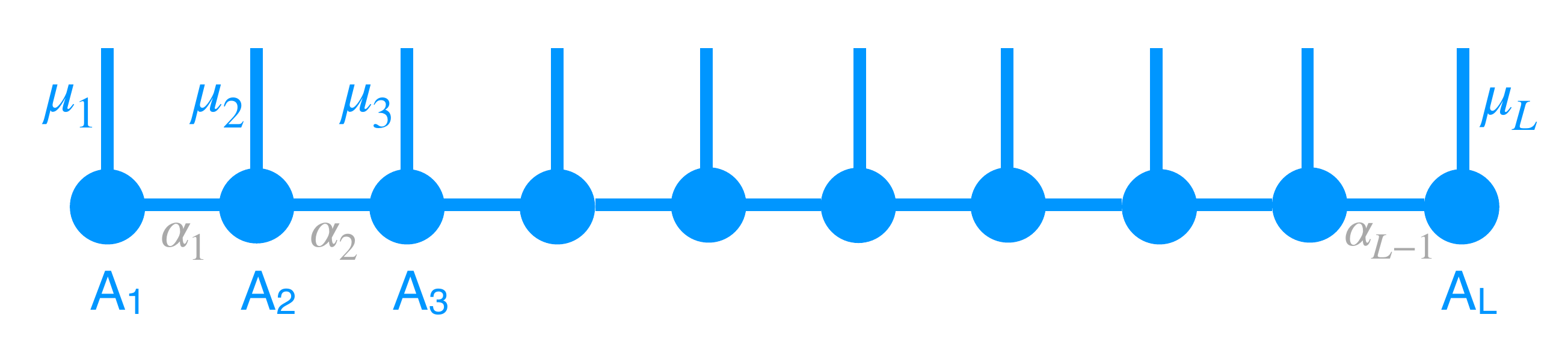}
\end{center}
\vspace*{-6mm}

\caption{Standard graphical representation of a tensor train. The basis functions $\varphi_k$ correspond to `sites' or `nodes'. Each node carries a matrix $A_k$ which is viewed as a function of the binary variable $\mu_k$ and the matrix indices ($\alpha_{k-1}$ and $\alpha_k$, except for the end nodes). These variables are represented by edges. Edges connecting two nodes correspond to the `virtual' variables $\alpha_k$ which appear in both adjacent nodes and are summed over.}
\label{F:TT}
\end{figure}

The tensor train decomposition of a tensor is not unique since the product of a matrix with its inverse can be inserted between each pair of matrices $A_{k}[\mu_k]$ and $A_{k+1}[\mu_{k+1}]$. Suitable ways to eliminate this nonuniqueness are discussed, e.g., in \cite{schollwoeck2011density} and \cite{holtz2012onmanifolds}. 

It is not difficult to observe that any tensor $(\psi_{\mu_1,...,\mu_L})\in \otimes_{i=1}^L \C^2$ can be brought into tensor-train format by successive singular value decompositions of matrix reshapes (see e.g. \cite{oseledets2009breaking,grasedyck2009hierarchical,schollwoeck2011density}).

In practice, successive SVD is impossible except for extremely small basis sets because the dimensions of the matrices $A_k$ blow up as $L$ gets large. It turns out that the dimensions $(r_k)_{1 \leq k \leq L-1}$ of these matrices for optimal tensor train representations of $(\Uppsi_{\mu_1 \dots \mu_L})$ are closely linked to the matrix reshape $(\Uppsi^{\mu_1\dots\mu_k}_{\mu_{k+1}\dots\mu_L}) \in \R^{2^k \times 2^{L-k}}$. 
Namely , for each $1 \leq k \leq L-1$ the dimension $r_k$ of the minimal tensor train representation~\eqref{eq:tensor-train-decomposition} (minimal in the sense of minimal ranks $(r_k)_{1 \leq k \leq L}$) is equal to the rank of the matrix $(\Uppsi_{\mu_{k+1}\dots\mu_L}^{\mu_1\dots\mu_k})$ \cite[Theorem~1]{holtz2012onmanifolds}. Simple examples showing that these ranks get huge even for Slater determinants are given in Section~\ref{sec:numerics}. 

It is then of central interest to understand how well, and how, a tensor $(\Uppsi_{\mu_1,...,\mu_L})$ can be approximated by a tensor train of prescribed dimension  
%
$(\tilde{r}_1,\dots,\tilde{r}_{L-1})$. An important first step that has been achieved in the mathematical literature are quasi-optimality estimates of the best approximation in terms of 
the singular values of the different reshapes of the tensor. More precisely, we have \cite{oseledets2009breaking,grasedyck2009hierarchical,hackbusch2012tensor,hackbush2014numerical} 
\[
\min\limits_{V \in \mathcal{M}_{\tilde{\mathbf{r}}}} \|\Uppsi - V\| \leq \sqrt{\sumlim{k=1}{L-1} \sumlim{j>\tilde{r}_k}{} {\sigma_j^{(k)}}^2} \leq \sqrt{L-1} \, \min\limits_{V \in \mathcal{M}_{\tilde{\mathbf{r}}}} \|\Uppsi - V\|,
\]
where for $1 \leq k \leq L-1$, $(\sigma_j^{(k)})_{1 \leq j \leq r_k}$ are the singular values of the reshape 
$(\Uppsi_{\mu_{k+1}\dots\mu_L}^{\mu_1\dots\mu_k})$ and $\mathcal{M}_{\tilde{\mathbf{r}}}$ is the space of tensor trains of dimension at most $(\tilde{r}_1,\dots,\tilde{r}_{L-1})$.

Studying the properties of the singular values of the reshaped matrix $(\Uppsi_{\mu_{k+1}\dots\mu_L}^{\mu_1\dots\mu_k})$ is thus essential to understand the accuracy of the tensor train approximation. 

Some works have recently appeared on the distribution of the singular values of tensors. In \cite{Hackbusch2017interconnection}, the authors have given partial answers on which sets of higher-order singular values are simultaneously possible for the different matrix reshpes. In \cite{griebel2019analysis}, the decay of the higher-order singular values of Sobolev functions is investigated for a continuous TT decomposition in which the variables $\mu_k$ in eq.~\eqref{eq:tensor-train-decomposition} belong to a continuous domain $\Omega \subset \R^n$. 
In \cite{shi2018numericalranks}, the authors study the distribution of the higher-order singular values of tensors built from polynomials. 

These results, while of interest in other contexts, do not apply to the TT representations of typical states in quantum chemistry as underlying QC-DMRG, where the tensors depend on binary variables, and arise via the nonlocal transformation described in section \ref{sec:Fock}
from states which are neither arbitrary (see, e.g., Proposition~\ref{prop:slater_ground_state_noninteracting_hamiltonian} and the subsequent remarks) nor polynomial nor of higher-order Sobolev regularity (we recall the well known cusp singularities caused by the Coulombic interactions). Our goal in the following is to study the distribution of the higher-order singular values of TT representations in quantum chemistry as used in QC-DMRG.

\paragraph{Notation} For $x \in \mathbb{R}^N$, $\|x\|$ denotes the Euclidean norm. We denote by $\binom{[k]}{j}$ for $1\leq j \leq k$ the set of multi-indices of cardinality $j$ (or $j$-combinations) such that
\[
\binom{[k]}{j} = \left\lbrace |\sigma|=j \ : \ 1  \leq \sigma(1) < \cdots < \sigma(j) \leq k \right\rbrace.
\]
The cardinality of the set $\binom{[k]}{j}$ is $\binom{k}{j}$. In an abuse of notation, we will assume that the elements of $\binom{[k]}{j}$ are in lexicographical order. For positive integers $j,k,\ell$, $\ell+\binom{[k]}{j}$ is the set of $j$ combinations of $\lbrace \ell+1,\dots,\ell+k \rbrace$. 

For a matrix $A$ and $\sigma \in \binom{[k]}{j}$, $A^\sigma$ is the submatrix of $A$ of rows with indices $ \lbrace \sigma(1),\dots,\sigma(j) \rbrace$. 
Similarly, $A_\sigma$ is the submatrix of $A$ of columns with indices $ \lbrace \sigma(1),\dots,\sigma(j) \rbrace$. The $k\times k$ identity matrix is denoted $\text{Id}_k$, and the $k\times k$ zero matrix by ${\bf 0}_{k\times k}$. 

\section{Singular values of the matricization of a Slater determinant}
\label{sec:maximal_virtual_bond}

In the following, we fix a single-particle basis set $\{\varphi_1,...,\varphi_L\}$ (i.e., an orthonormal set of functions in $H^1(\R^3\times\{\pm \frac{1}{2}\})$) and assume that $\Psi$ is a Slater determinant,
\begin{equation}
\label{eq:tensor}
\Psi = |\psi_1,...,\psi_N\rangle,
\end{equation}
where the orbitals $\psi_i$ belong to $\mathrm{Span}\{\varphi_j\}_{j=1}^L$ and $\{\psi_i\}_{i=1}^N$ is an orthonormal family.
By assumption, there exists a partial isometry $U \in \R^{N \times L}$ (\emph{i.e.} $UU^T = \mathrm{Id}_N$) such that $\psi_i=\sum_{j=1}^L U_{ij}\psi_j$, i.e. 
\begin{equation}
\label{eq:definition_U}    
\begin{pmatrix} \psi_1 \\ \vdots \\ \psi_N \end{pmatrix} = U \begin{pmatrix} \varphi_1 \\[2mm] \vdots \\[2mm] \varphi_L \end{pmatrix}.
\end{equation}
The rows of $U$ are thus the coefficients of the orbitals $\psi_i$ with respect to the basis $\{\varphi_1,...,\varphi_L\}$. The columns of $U$ are denoted by $u_j \in \R^N$, $j=1,\dots,L$, i.e.
\begin{equation}
U = \Big( u_1 \dots u_L \Big).
\end{equation}
Hence the submatrices $V_k$ and $W_k$ defined in Equation~\ref{eq:decomposition_V} below are $V_k = \Big( u_1 \cdots u_k \Big)$ and $W_k = \Big( u_{k+1} \cdots u_L \Big)$.

The occupation representation \eqref{occrep1}--\eqref{occrep2} of the orbital $\psi_i$, which we still denote $\psi_i$, is then
$$
  \psi_i = U_{i1} \Phi_{(10...0)} + U_{i2} \Phi_{(010...0)} + ... + U_{iL} \Phi_{(0...01)}.
$$
In this representation, the orbitals $\psi_i$ can thus be interpreted as functions on the set of nodes of the tensor network.

We claim that the occupation representation of the Slater determinant $\Psi$ is
$$
  \Uppsi = \sum_{\mu_1,...,\mu_L=0}^1 \Uppsi_{\mu_1 ... \mu_L} \, \Phi_{(\mu_1,...,\mu_L)}
$$
with the coefficients
\begin{equation}\label{slater_coefficient}
   \Uppsi_{\mu_1...\mu_L} = \begin{cases} 0 & \mbox{if }\sum_{i=1}^L\mu_i \neq N \\
                      \det(u_{i_1} \cdots u_{i_N}) & \mbox{otherwise,} \end{cases}
\end{equation}
where, for any given $(\mu_1,...,\mu_L)$ with $\sum_i\mu_i=N$, $i_1<...<i_N$ are the indices such that $\mu_{i_k}=1$. To see this, we expand the Slater determinant \eqref{eq:tensor} as follows, denoting the group of permutations $\sigma\, : \, \{1,...,N\}\to\{1,...,N\}$ by $S_N$ and the signature of $\sigma$ by $\epsilon(\sigma)\in\{\pm 1\}$: 
\begin{eqnarray*}
  \Psi &=& \Big|\sum_{j_1=1}^L U_{1j_1}\varphi_{j_1}, \, ... \, , \, \sum_{j_N=1}^L U_{N j_N}\varphi_{j_N}\Big\rangle \\
       &=& \sum_{j_1,...,j_N=1}^L U_{1j_1}\cdot ... \cdot U_{Nj_N} |\varphi_{j_1},...,\varphi_{j_N}\rangle \\     &=& \sum_{1\le j_1<...<j_N\le L}^{\textcolor{white}{L}} \underbrace{\sum_{\sigma\in S_N} \epsilon(\sigma) U_{1j_{\sigma(1)}}\cdot ... \cdot U_{Nj_{\sigma(N)}}}_{=\det(u_{j_1}\cdots u_{j_N})} |\varphi_{j_1}, ... , \varphi_{j_N}\rangle. 
\end{eqnarray*}
The formula for the coefficients is now immediate from \eqref{occrep1}--\eqref{occrep2}. 

\begin{assumption}
\label{assumption:invertibility}
Let $1 \leq k \leq L$, let $V_k \in \R^{N \times k}$ and $W_k \in \R^{N \times (L-k)}$ be the matrices such that
\begin{equation}
\label{eq:decomposition_V}
U = \Big( V_k \  W_k \Big).
\end{equation}
We assume that $V_k$ and $W_k$ are full-rank matrices.
\end{assumption}

\begin{theo}
\label{thm:sing_values_TT_of_slater}
Let $1 \leq k \leq L-1$ and let $\sigma_1 \geq \dots \geq \sigma_d$ be the nonzero singular values of $\Uppsi^{\mu_1,\dots,\mu_{k}}_{\mu_{k+1},\dots,\mu_L}$. Then under Assumption~\ref{assumption:invertibility} there are exactly $d=\min(2^k,2^N,2^{L-k})$ nonzero singular values. Moreover these singular values satisfy
\begin{equation}
\label{eq:singular_values_reshape_property}
    \forall \, 1 \leq j \leq d, \ \sigma_j^2 \sigma_{d-j}^2 = p(k,L,N)
\end{equation}
for some $j$-independent constant $p(k,L,N)$. Explicitly,
\begin{equation} \label{prefac}
   p(k,L,N) = 
    \begin{cases} 
    \det(V_k^T V_k) \det(W_kW_k^T), & \text{ if } k \leq N \\
    \det(V_k V_k^T) \det(W_kW_k^T), & \text{ if } N \leq k \leq L-k \\
    \det(V_k V_k^T) \det(W_k^TW_k), & \text{ if } L-k \leq k \leq L-1. 
    \end{cases}
\end{equation}
\end{theo}

We remark that the inversion symmetry is directly visible in the logarithmic plots of singular value distributions in Figure \ref{fig:slater-comparison}, as symmetry of all graphs with respect to inversion at the mid-point.

Theorem~\ref{thm:sing_values_TT_of_slater} in particular gives the following universal upper bound on the bond dimension of the tensor train representation of any $N$-electron Slater determinant.

\begin{cor}
\label{cor:maximal_bond_dimension}
The bond dimension of a minimal rank tensor train representation of an $N$-particle Slater determinant $\Psi$ is at most $2^N$.
\end{cor}

It is interesting to notice that the bond dimension is independent of the number of one-particle basis functions considered. Moreover this bound is optimal:

\begin{cor} \label{cor:example}
For $L=2N$ and $\psi_k = \frac{1}{\sqrt{2}}(\varphi_k+\varphi_{k+N})$ for $k=1,...,N$, the bond dimension of the minimal tensor train representation of the Slater determinant $|\psi_1,...,\psi_N\rangle$ is exactly $2^N$.
\end{cor}
 
Corollary~\ref{cor:maximal_bond_dimension} does not require Assumption~\ref{assumption:invertibility}, since the singular values depends continuously on the coefficients of the matrix. This result is known in the physics community; it is implicit in \cite{schneider2014tensor} and explicit in \cite{silvi2013matrix}). 


Our results on the singular values of the reshaped tensor $(\Uppsi^{\mu_1,\dots,\mu_{k}}_{\mu_{k+1},\dots,\mu_L})$ show that the bond dimension of a noninteracting state is \emph{not} small. 
Moreover their decay is related to the order in which the one-particle basis functions are labelled. In fact, reordering the one-particle basis functions spectacularly impacts the behavior of the tail of the singular value distributions, as will be shown in Figure~\ref{fig:slater-comparison}. 
Moreover since the prefactor $p(k,L,N)$ governing the singular values is explicitly known, it provides a natural way to choose the ordering of the one-particle basis functions in order to optimize the tail distribution of the singular values. 
This idea is explained thoroughly in Section~\ref{subsec:best_prefactor_order}.

Theorem~\ref{thm:sing_values_TT_of_slater} enables one, via Weyl's inequality, to bound the singular values of a tensor which is a sum of Slater determinants, which can be seen as a basic model of a correlated state.

\begin{theo}
\label{cor:strongly_correlated_hosv}
Let $(\alpha_I)$ be complex coefficients such that $\sumlim{I}{} |\alpha_I|^2 =1$ and let
\(
\Psi = \sumlim{\substack{I \subset \lbrace 1,\dots,L \rbrace \\ |I|=N}}{} \alpha_I \Psi_{[I]},
\)
where $\Psi_{[I]} = | \psi_{i_1},...,\psi_{i_N} \rangle$, $I=(i_1,...,i_N)$.

Let $\Uppsi_{\mu_1,\dots,\mu_L}$ (resp. $\Uppsi^{(I)}_{\mu_1,\dots,\mu_L}$) be the tensor representation of $\Psi$ (resp. $\Psi_{[I]}$) in the basis of Slater determinants $\Phi_{[I]}$. Let $(\sigma_j)$ be the singular values of $(\Uppsi^{\mu_1,\dots,\mu_k}_{\mu_{k+1},\dots,\mu_L})$, and $(\sigma_j^{(I)})$ be the singular values of $\left(\Uppsi^{(I)}\right)^{\mu_1,\dots,\mu_k}_{\mu_{k+1},\dots,\mu_L}$.

Then for all $1 \leq j \leq 2^{\min(k,L-k)}$ and any $j_I \in \N$ such that $\sumlim{\substack{I \subset \lbrace 1,\dots,L \rbrace \\ |I|=N}}{} j_I = j$, we have
\begin{equation}
    \sigma_j \leq \sumlim{\substack{I \subset \lbrace 1,\dots,L \rbrace \\ |I|=N}}{} |\alpha_I| \sigma^{(I)}_{j_I}.
\end{equation}
\end{theo}

Since the singular values of a correlated state are upper-bounded by linear combinations of singular values of a single Slater determinant, one can optimize the ordering of the basis functions to lower the upper bound. 
The corresponding scheme will be outlined in Section~\ref{subsec:best_weighted_prefactor}.

\section{Ordering the one-particle basis functions} \label{sec:ordering}

We now come to a central issue in QC-DMRG, and in tensor approximations generally: {\it How to choose the topology of the network?}
When the network is a tensor train associated with a fixed one-particle basis $\{\varphi_1,...,\varphi_L\}$, as described in Figure \ref{F:TT}, this question reduces to: {\it How to choose the ordering of the basis?}  

\subsection{Canonical order}
As explained in Section~\ref{sec:FCI}, in practice the single-particle basis typically consists of the low-lying eigenfunctions of the Fock operator. The simplest method, and the one used in early QC-DMRG calculations, is to order the orbitals simply according to their Hartree-Fock eigenvalues. This ordering is known as {\it canonical order}. 

\subsection{Fiedler order}
A significant improvement was achieved in the pioneering work \cite{barcza2011quantum} which introduced the Fiedler ordering; it brings into play concepts from quantum information theory and spectral graph theory. A precursor of this method can be found in \cite{rissler2006measuring} 
where orbitals were re-ordered numerically so as to promote smaller bandwidth of the mutual information matrix. 

The mutual information matrix $(I\! M_{ij})_{1 \leq i,j \leq L}$ is defined by 
\begin{equation}
\label{eq:mutual_information}
I\! M_{ij} = (1-\delta_{ij})(S_i^{(1)}+S_j^{(1)}-S_{ij}^{(2)}),
\end{equation}
where $S_i^{(1)}$ and $S_{ij}^{(2)}$ are respectively the von Neumann entropies of the one-orbital density matrices $\rho_i^{(1)}$ and the two-orbital density matrices $\rho_{ij}^{(2)}$(see eqs.~\eqref{RDM1},~\eqref{RDM2} below); recall that the von Neumann entropy of a density matrix is $S(\rho)=-\mbox{tr}\, \rho \, \log \rho$. The matrix element $I\! M_{ij}$ can also be interpreted as the minus the relative entropy (alias Kullback-Leibler divergence) between $\rho^{(2)}_{ij}$ and $\rho^{(1)}_i \otimes \rho^{(1)}_j$:
$$
   I\! M_{ij} = KL(\rho^{(2)}_{ij}||\rho^{(1)}_i\otimes\rho_j^{(1)}), \;\;\; \mbox{ with }KL(\rho || \rho') = \mbox{tr} \, \rho(\log\rho - \log\rho').  
$$ 
The \emph{rationale} behind the mutual information is that it captures the correlation between two sites for the considered state; note that the above KL divergence vanishes if and only if $\rho^{(2)}_{ij}$ factors into $\rho^{(1)}_i\otimes\rho^{(1)}_j$.    
Hence in a tensor train -- or in a more general tensor network -- one wants to keep sites with high mutual information quite close to each other. 

The one-orbital and two-orbital RDM (reduced density matrices) of a state $\Uppsi\in\otimes_{i=1}^L\C^2$ are defined as the partial trace over all the remaining orbitals. Namely
\begin{equation}\label{RDM1}
\rho_i^{(1)} = \mathrm{Tr}_{1,\dots,\not i,\dots,L} |\Uppsi \rangle \langle \Uppsi|
\end{equation}
and
\begin{equation}\label{RDM2}
\rho_{ij}^{(2)} = \mathrm{Tr}_{1,\dots,\not i,\dots, \not j, \dots,L} |\Uppsi \rangle \langle \Uppsi|
\end{equation}
where $|\Uppsi\rangle \langle\Uppsi|$ denotes the orthogonal projector onto $\Uppsi$. Thus $\rho_i^{(1)}$ and $\rho_{ij}^{(2)}$ are nonnegative unit-trace operators on, respectively, the single-site tensor factor $\C^2=\mbox{Span}\{\chi_0,\chi_1\}$ and the two-site space $\C^2\otimes\C^2=\mbox{Span}\{\chi_0\otimes\chi_0,\chi_0\otimes\chi_1,\chi_1\otimes\chi_0,\chi_1\otimes\chi_1\}$. 

Even in the case of Slater determinants, the innocent looking two-orbital RDM \eqref{RDM2} exhibits considerable complexity. This is because it does not describe entanglement between {\it particles} (which is absent for Slater determinants) but entanglement between {\it sites}. Moreover this entanglement depends on the ordering of the nodes of the network. To document all this rigorously, we worked out the two-orbital RDM explicitly. The result appears to be new.

\begin{prop}[One-orbital and two-orbital RDM of a Slater determinant]
\label{prop:orbital_density_matrix}
The one-orbital density matrix of a Slater determinant $\Psi$ (eq.~\eqref{eq:tensor}) is the $2 \times 2$ matrix 
\[
\rho^{(1)}_i = \begin{pmatrix} 1-\|u_i\|^2 & 0 \\ 0 & \|u_i\|^2 \end{pmatrix}.
\]
For $1 \leq i < j \leq L$, the two-orbital density matrix of a Slater determinant $\Psi$ is the $4 \times 4$ matrix
\[
\rho^{(2)}_{ij} = \begin{pmatrix} 1-\|u_i\|^2-\|u_j\|^2+G & & & \\ & \|u_i\|^2-G & \rho^{(2)}_{ij}(01,10) &  \\ & \rho^{(2)}_{ij}(10,01) & \|u_j\|^2-G &  \\  & & & G \end{pmatrix},
\]
where $G = \|u_i\|^2\|u_j\|^2 - \langle u_i,  u_j\rangle^2$ and 
\begin{equation*}
\rho^{(2)}_{ij}(10,01) = \rho^{(2)}_{ij}(01,10) = \sum\limits_{k=0}^{j-i-1} \sum\limits_{\gamma \in \binom{[i+1:j-1]}{k}} (-1)^{k+1}
\det \begin{pmatrix}
\mathbf{0}_{k+1 \times k+1} & U_{\gamma \cup \lbrace j \rbrace}^T \\
U_{\lbrace i \rbrace \cup \gamma} & \mathrm{Id}_N - U_{\gamma^c} U_{\gamma^c}^T \\ 
\end{pmatrix} ,
\end{equation*}
where $\gamma^c = [i+1:j-1]\setminus \gamma$.
\end{prop}

The proof of Proposition~\ref{prop:orbital_density_matrix} can be found in Section~\ref{subsec:proof-prop-RDM-Slater}. We give three special cases in which the formula for the off-diagonal elements simplifies.

\begin{prop} \label{prop:more_on_RDMs}
a) For $j=i+1$, the formula of the two-orbital density matrix simplifies to 
\[
\rho^{(2)}_{ij} = \begin{pmatrix} 1-\|u_i\|^2-\|u_j\|^2+G & & & \\ & \|u_i\|^2-G & \langle u_i,\, u_j\rangle &  \\ & \langle u_i, \, u_j\rangle & \|u_j\|^2-G &  \\  & & & G \end{pmatrix}.
\]
b) For $j=i+2$, the off-diagonal elements of the two-orbital density matrix depend on $u_{i+1}$:
\[
\rho^{(2)}_{ij}(10,01)=\rho^{(2)}_{ij}(01,10) = \langle u_j, \, \Bigl(\mathrm{Id}_N - 2(||u_{i+1}||^2 - u_{i+1}u_{i+1}^T)\Bigr) u_i\rangle.
\]
c) For $N=2$ and arbitrary $i$, $j$ with $i<j$, 
\[
\rho^{(2)}_{ij}(10,01)=\rho^{(2)}_{ij}(01,10) = \sum_{k \, : \, i<k<j} \langle u_j, \, \Bigl(\mathrm{Id}_N - 2(||u_{k}||^2 - u_{k}u_{k}^T)\Bigr) u_i\rangle.
\]
\end{prop}
Here $\langle u_i,\, u_j\rangle$ denotes the scalar product $u_i^T u_j$. Thus in general, the off-diagonal term of $\rho^{(2)}_{ij}$ depends on the occupancy of the orbitals between $i$ and $j$, and hence so does the mutual information $I\! M_{ij}$.

Once the mutual information matrix is computed, the Fiedler ordering is obtained as follows. It is a simple fact that the entries of the mutual information matrix are nonnegative (see, e.g., \cite{carlen2014remainder}). Hence it can be interpreted as weighted adjacency matrix of the complete graph of the tensor network. The graph Laplacian $\mathcal{L}$ defined by
\[
\mathcal{L}_{ij} = \begin{cases} \sumlim{k=1}{L} I\! M_{ik} & \text{if } i=j \\
-I\! M_{ij} & \text{else} \end{cases}
\]
is computed. 
The second eigenvector of the graph Laplacian is called the \emph{Fiedler vector} and ordering its entries according to its values gives the so-called \emph{Fiedler ordering}.
In spectral graph theory (see, e.g., \cite{vonLuxburg2007spectralgraph}), the Fiedler order enables one to regroup nodes of a graph according to their connected components. 
Hence its use for ordering the one-particle basis functions makes intuitive sense if the mutual information is the right choice of metric on their complete graph. 

We remark that the mutual information matrix, and hence the Fiedler ordering, is not entirely independent of the original ordering, as the latter can influence the two-orbital RDMs (see Proposition~\ref{prop:more_on_RDMs}). This phenomenon illustrates the subtleties of applying the TT format, which originated in spin chain theory, to quantum chemistry. 

\subsection{Best prefactor order}
\label{subsec:best_prefactor_order}

Theorem~\ref{thm:sing_values_TT_of_slater} suggests a new scheme to find a good ordering of the one-particle functions $\varphi_i$. Assuming that the largest singular value of a matrix reshape of $\Psi$ remains large for all reorderings, the tail of the singular value distribution can be lowered by finding the ordering that yields the smallest prefactor in Equation~\eqref{eq:singular_values_reshape_property}. 
To be more precise, one needs to solve the following discrete optimization problem:
\begin{multline}
\label{eq:discrete-optimization-prefactor-pb}    
\text{Find } \sigma \in \tbinom{[L]}{N} \text{ such that: } \\ \det\left(U_\sigma U_\sigma^T\right) \det\left(\mathrm{Id}_N - U_\sigma U_\sigma^T\right) \leq \det\left(U_\tau U_\tau^T\right) \det\left(\mathrm{Id}_N - U_\tau U_\tau^T\right) \quad \forall\, \tau \in \tbinom{[L]}{N}.
\end{multline}
We call its solution the {\it best prefactor order}. Being an optimum over permutations, this order has the desirable feature of being independent of the initial ordering of the basis. 

In the numerical tests below, \eqref{eq:discrete-optimization-prefactor-pb} can be solved exactly, but the complexity grows as $\binom{L}{N}$ and so for larger values of $N$ and $L$ (say, $N=20$ and $L=40$) solving it exactly is out of reach. 
However, an approximate solution can always be computed using a simulated annealing algorithm (see Algorithm~\ref{algo:new_scheme}).
We call this approximate solution {\it approximate best prefactor}. 
The Fiedler order, which is widely used in practice and described above, typically gives a better order of the orbitals than a random ordering guess. It can thus be used as a starting point for Algorithm~\ref{algo:new_scheme}. 

\begin{algorithm}[!h]
\DontPrintSemicolon
 \KwData{$U$,\ $i_\mathrm{max}$,\ $\mathrm{AS}$,\ $\mathrm{VS}, \tau_0, \lambda$ \tcp*[r]{$\mathrm{AS}$: Active Space, $\mathrm{VS}$:Virtual Space, $\tau_0$: initial temperature, $\lambda$: temperature decay rate} }
 \KwResult{$\mathrm{AS}, \ \mathrm{VS}$}
 $i=0$\;
 $\tau=\tau_0$\;
 \While{$i<i_\mathrm{max}$}{
 $\tau = \lambda \tau$\;
 $j=\mathrm{rand}(\mathrm{AS})$\;
 $k=\mathrm{rand}(\mathrm{VS})$\;
 $\mathrm{TAS} = \mathrm{AS} \cup \lbrace k \rbrace \setminus \lbrace j \rbrace$\tcp*{$\mathrm{TAS}$: test active space}\; 
 $\mathrm{TVS} = \lbrace 1,\dots,L \rbrace \setminus \mathrm{TAS}$\tcp*{$\mathrm{TVS}$: test virtual space}\;
  \uIf{$\det(U_{\mathrm{TAS}} U_{\mathrm{TAS}}^T) \det(U_{\mathrm{TVS}} U_{\mathrm{TVS}}^T) < \det(U_{\mathrm{AS}} U_{\mathrm{AS}}^T) \det(U_{\mathrm{VS}} U_{\mathrm{VS}}^T)$}{
   $\mathrm{AS} = \mathrm{TAS}$\;
   $\mathrm{VS} = \mathrm{TVS}$\;
   }
  \ElseIf{$\exp\left( \frac{\det(U_{\mathrm{TAS}} U_{\mathrm{TAS}}^T) \det(U_{\mathrm{TVS}} U_{\mathrm{TVS}}^T) - \det(U_{\mathrm{AS}} U_{\mathrm{AS}}^T) \det(U_{\mathrm{VS}} U_{\mathrm{VS}}^T)}{\tau} \right) > \mathrm{rand}()$}{
   $\mathrm{AS} = \mathrm{TAS}$\;
   $\mathrm{VS} = \mathrm{TVS}$\;
   }
  $i=i+1$\;
 }
 \caption{Simulated annealing ordering scheme}
 \label{algo:new_scheme}
\end{algorithm}
\subsection{Best weighted prefactor order}
\label{subsec:best_weighted_prefactor}
For a correlated state, given by a sum of Slater determinants, the behavior of the singular values is more intricate. 
Nonetheless, by Theorem~\ref{cor:strongly_correlated_hosv} and adapting the ordering strategy for Slater determinants described previously, it is possible to optimize the ordering to improve the decay the singular values.
More precisely, for a state \(
\Psi = \sumlim{\substack{I \subset \lbrace 1,\dots,L \rbrace \\ |I|=N}}{} \alpha_I \Psi_{I},
\)
where $\Psi_{I} = | \psi_{i_1},...,\psi_{i_N} \rangle$, $I=(i_1,...,i_N)$, we propose to choose 
\[
\sigma = \argmin\limits_{\tau \in S_L} \sumlim{\substack{I \subset \lbrace 1,\dots,L \rbrace \\ |I|=N}}{} |\alpha_I| p_I(\tau),
\]
with $p_I(\tau)$ being the prefactor of the Slater determinant $\Psi_I$ defined in~\eqref{prefac}.
We call the solution \emph{best weighted prefactor order}. Even for weakly correlated states, \emph{i.e.} when one Slater determinant is dominant in the expansion, numerical tests (see the next section) show that it is preferable to optimize the ordering by taking into account all the determinants as above, instead of focusing only on the dominant one.

\subsection{An example: minimal-basis H$_2$}
\rm 
To illustrate the different ordering methods, we now apply them to the minimal-basis H$_2$ wavefunction  introduced in Example 1.1 at the end of section~\ref{sec:Fock}. Starting point is the occupation representation \eqref{H2Slateroccrep} of the state \eqref{H2Slater}. In this example, $L=4$, so -- recalling that all coeffients were taken to be real -- the occupation representation $\Uppsi$ belongs to $\otimes_{i=1}^4 \R^2$, i.e. it is a four-index tensor $(\Uppsi_{\mu_1\mu_2\mu_3\mu_4})$ with $2^4$ real components. We will compute the singular values of the matrix re-shape $\Uppsi^{\mu_1\mu_2}_{\mu_3\mu_4}\in \R^{2^2 \times 2^2}$, for the different orderings of the basis set delivered by all the above ordering schemes. To avoid degenerate cases we assume that all coefficients $c$, $s$, $c'$, $s'$ in \eqref{H2Slater} are nonzero.    
\\[2mm]
{\bf Canonical order.} We abbreviate the single-particle basis states as $\{A\uparrow, \, A\downarrow, \, B\uparrow, \, B\downarrow\}$. Directly from \eqref{H2Slateroccrep} we see that with respect to the canonical order in which the bonding orbital with either spin comes first,
\begin{equation}\label{H2canonical}
            A\uparrow  \; A\downarrow \; B\uparrow \; B\downarrow ,
\end{equation}
the re-shape $\Uppsi^{\mu_1\mu_2}_{\mu_3\mu_4}$ is
\begin{center}
\begin{tabular}{c | c c c c}
              $_{\mu_1\,\mu_2} \!\!\! \diagdown \!\! ^{\mu_3\,\mu_4}$\hspace*{-2mm} & 00 & 01 & 10 & 11  \\
              \hline
              00\textcolor{white}{00} &   &   &   & $ss'$ \\
              01\textcolor{white}{00} &   &   & $-sc$ &   \\
              10\textcolor{white}{00} &   & $cs'$ &   &   \\
              11\textcolor{white}{00} & $cc'$ &   &   &   \\  
              \end{tabular}
\end{center}
The singular values are 
$$
        (cc')^2, \, (cs')^2, \, (sc')^2, \, (ss')^2
$$
and the rank of the matrix re-shape is $4$. 
 
As an off-spring we recover the inversion symmetry of the distribution of singular values predicted by Theorem~\ref{thm:sing_values_TT_of_slater}, since $cc' \cdot ss' = cs' \cdot sc'$. 
\\[2mm]
{\bf Fiedler order.} We begin by working out the one- and two-orbital density matrices and the  corresponding entropies. The one-orbital quantities are elementary to compute, they are  
\[
  \rho^{(1)}_{A\uparrow} = \begin{pmatrix} s^2 &  0 \\ 0 & c^2 \end{pmatrix}, \;\;\;
  \rho^{(1)}_{A\downarrow} = \begin{pmatrix} s'{^2} &  0 \\ 0 & c'{^2} \end{pmatrix}, \;\;\;
  \rho^{(1)}_{B\uparrow} = \begin{pmatrix} c^2 &  0 \\ 0 & s^2 \end{pmatrix}, \;\;\;
  \rho^{(1)}_{B\uparrow} = \begin{pmatrix} c'{^2} &  0 \\ 0 & s'{^2} \end{pmatrix}. \;\;\;
\]
It follows that 
\begin{align*}
  s^{(1)}_{A\uparrow} &= s^{(1)}_{B\uparrow} = -c^2\log c^2-s^2\log s^2 =: s_\uparrow \in (0,1], \\
  s^{(1)}_{A\downarrow} &= s^{(1)}_{B\downarrow} = -c'{^2}\log c'{^2}-s'{^2}\log s'{^2} =: s_\downarrow \in (0,1].
  \end{align*}
As regards the two-orbital density matrices, we find after some calculation that 
\begin{scriptsize}
\[ \hspace*{-5mm}
  \rho^{(2)}_{A\uparrow A\downarrow} = \begin{pmatrix}
              s^2 s'{^2}\hspace*{-4mm} &  &   &   \\
              & s^2 c'{^2}\hspace*{-4mm}  &   &   \\
              &  & c^2 s'{^2}\hspace*{-4mm}   &   \\
              &  &  &  c^2 c'{^2}   \end{pmatrix}\! , \;  
  \rho^{(2)}_{B\uparrow B\downarrow} = \begin{pmatrix}
              c^2 c'{^2}\hspace*{-4mm} &  &   &   \\
              & c^2 s'{^2}\hspace*{-4mm}  &   &   \\
              &  & s^2 c'{^2}\hspace*{-4mm}   &   \\
              &  &  &  s^2 s'{^2}   \end{pmatrix}\! , \;  
  \rho^{(2)}_{A\uparrow B\downarrow} = \begin{pmatrix}
              s^2 c'{^2}\hspace*{-4mm} &  &   &   \\
              & s^2 s'{^2}\hspace*{-4mm}  &   &   \\
              &  & c^2 c'{^2}\hspace*{-4mm}   &   \\
              &  &  &  c^2 s'{^2}   \end{pmatrix}\! , \;  
  \rho^{(2)}_{A\downarrow B\uparrow} = \begin{pmatrix}
              c^2 c'{^2}\hspace*{-4mm} &  &   &   \\
              & c^2 s'{^2}\hspace*{-4mm}  &   &   \\
              &  & s^2 s'{^2}\hspace*{-4mm}   &   \\
              &  &  &  s^2 c'{^2}   \end{pmatrix}\! .  
 \]
 \end{scriptsize}
 
 \noindent
It follows that $S^{(2)}=-\mbox{tr} \, \rho^{(2)}\log \rho^{(2)}=:S_{\uparrow\downarrow}$ is the same for all four matrices. Moreover writing out the above trace and using $c^2+s^2=c'{^2}+s'{^2}=1$ we find that
\begin{equation}\label{cancellation}
   S_{\uparrow\downarrow} = s_\uparrow + s_\downarrow.
\end{equation}
The two remaining two-orbital RDMs contain off-diagonal terms. We find using Proposion~\ref{prop:more_on_RDMs}~c) that
\begin{scriptsize}
\[
  \rho^{(2)}_{A\uparrow B\uparrow} = \begin{pmatrix}
              0  &   &   &   \\
              &  s^2 & -cs(c'{^2}-s'{^2})   &   \\
              &  -cs(c'{^2}-s'{^2}) & c^2 &   \\
              &  &  &  0   \end{pmatrix}\! , \;\;\;
  \rho^{(2)}_{A\downarrow B\downarrow} = \begin{pmatrix}
              0  &   &   &   \\
              &  s'{^2} & c's'(c{^2}-s{^2})   &   \\
              &  c's'(c{^2}-s{^2}) & c'{^2} &   \\
              &  &  &  0   \end{pmatrix}\! .
\]
\end{scriptsize}

\noindent
We denote the associated entropies by $S^{(2)}_{A\uparrow B\uparrow}=:S_{\uparrow\uparrow}$, $S^{(2)}_{A\downarrow B\downarrow}=:S_{\downarrow\downarrow}$. The mutual information matrix and graph Laplacian are thus, using the vanishing of all nearest-neighbour elements of $I\! M$ by \eqref{cancellation} and denoting $a:=2s_{\uparrow}-S_{\uparrow\uparrow}$, $b:=2s_{\downarrow}-S_{\downarrow\downarrow}$, 
\begin{center}
$I\! M =$ 
\begin{tabular}{c | c c c c}
                  & $A\uparrow$ & $A\downarrow$ & $B\uparrow$ & $B\downarrow$  \\
              \hline
              $A\uparrow$   &  0  & 0  & a  & 0 \\
              $A\downarrow$ &  0  & 0  & 0  & b  \\
              $B\uparrow$   &  a  & 0  & 0  & 0 \\
              $B\downarrow$ &  0  & b  & 0  & 0  \\  
              \end{tabular} 
    $\,$  \textcolor{white}{,} $\;\;\;\;$ ${\mathcal L} =$  
\begin{tabular}{c | r r r r}
                  & $A\uparrow$ & $A\downarrow$ & $B\uparrow$ & $B\downarrow$  \\
              \hline
              $A\uparrow$   &  a  & 0  & -a  & 0 \\
              $A\downarrow$ &  0  & b  & 0  & -b  \\
              $B\uparrow$   &  -a  & 0  & a  & 0 \\
              $B\downarrow$ &  0  & -b  & 0  & b  \\  
              \end{tabular} 
    $\,$ \textcolor{white}{.} 
\end{center}
To determine the Fiedler ordering we need to find the second eigenvector of the graph Laplacian, alias Fiedler vector. The first eigenvector is always, by construction, the constant vector, with eigenvalue $0$. For the above $\mathcal{L}$, by inspection the remaining eigenvalues are $0$, $2a>0$, $2b>0$, with eigenvectors $(1 , -1 ,  1 , -1)$, $(1 , 0 , -1 , 0)$, $(0 , 1 , 0 , -1)$. The second eigenvector is thus $(1 , -1 , 1 , -1)$. It follows that the Fiedler ordering is
\begin{equation} \label{H2Fiedler}
                  A\uparrow  \; B\uparrow \; A\downarrow \; B\downarrow 
\end{equation} 
(up to re-ordering the orbitals in the left block, re-ordering the orbitals in the right block, and flipping the two blocks; none of this affects the singular values). 
The matrix re-shape $\Uppsi^{\mu_1\mu_2}_{\mu_3\mu_4}$ with respect to this ordering is
\begin{center}
\begin{tabular}{c | c c c c}
              $_{\mu_1\,\mu_2} \!\!\! \diagdown \!\! ^{\mu_3\,\mu_4}$\hspace*{-2mm} & 00 & 01 & 10 & 11  \\
              \hline
              00\textcolor{white}{00} &  0 &   &   &      \\
              01\textcolor{white}{00} &    & $ss'$  & $sc'$ &   \\
              10\textcolor{white}{00} &    & $cs'$ & $cc'$  &   \\
              11\textcolor{white}{00} &    &   &   & 0  \\  
              \end{tabular}
\end{center}
Since the middle block is the rank-1 matrix $\begin{pmatrix} c \\ s \end{pmatrix}\begin{pmatrix} c' & s' \end{pmatrix}$, the singular values are 
$$
        1, \, 0, \, 0, \, 0
$$
and the rank of the matrix re-shape is $1$. We see that the Fiedler order has dramatically improved the decay of the singular values. 
\\[2mm]
{\bf Best prefactor order.} First we note that, for arbitrary Slater determinants, the prefactor \eqref{prefac} -- just like the singular value distribution -- is invariant under switching the left and right blocks and re-ordering the states within each block. Thus we only need to consider the two orderinges \eqref{H2canonical} and \eqref{H2Fiedler}. With respect to the former respectively the latter, the partial isometry $U = (V \, | \, W) \in \R^{2\times 4}$ which represents the orbitals is
$$
    (V \, | \, W) = 
  \left(
  \begin{array}{cc|cc}
   c & 0 & s & 0 \\
   0 & c' & 0 & s' \\
   \end{array}
  \right)
 \;\;\; \mbox{respectively} \;\;\;
    (V \, | \, W) = 
  \left(
  \begin{array}{cc|cc}
   c & s & 0 & 0 \\
   0 & 0 & c' & s' \\
   \end{array}
  \right) .
$$ 
The prefactor is $p=\det(V^T V)\det(W^T W)=(\det V \, \det W)^2$. Thus it equals $(cc'ss')^2>0$ for the first and $0$ for the second matrix. Hence the best prefactor ordering is the second ordering, 
\begin{equation} \label{H2bestpref}
                  A\uparrow  \; B\uparrow \; A\downarrow \; B\downarrow. 
\end{equation}
The resulting singular value distribution of the matrix re-shape was already computed above, it is
$$
     1, \, 0, \, 0, \, 0
$$
and in particular the rank of the matrix re-shape is $1$. We see that, just like the Fiedler ordering, the best prefactor ordering dramatically improves the decay of the singular values. Moreover in this example we found that the Fiedler and best prefactor methods are {\it exactly equivalent!}
\section{Numerical comparison of the different ordering methods}
\label{sec:numerics}

\subsection{Tests on Slater determinants} \label{sec:testslat}

As a first comparison we tested all the above ordering schemes  
on Slater determinants \eqref{eq:tensor}. The results are given in Figure~\ref{fig:slater-comparison}. This figure shows the distribution of the singular values of the re-shape $\Psi^{\mu_1 \dots \mu_{L/2}}_{\mu_{L/2+1}\dots\mu_L}$ (which typically has the highest rank), for different ordering methods and averaged over 400 simulations. The partial isometry $U$ was obtained by taking the first $N$ rows of the orthogonal matrix $Q$ of the QR-decomposition of a random matrix of size $L \times L$ with i.i.d. standard normal entries. The ensuing Slater determinant $\Psi$ may thus be viewed as the {\it ground state of a random one-body Hamiltonian} acting on the $N$-body space $\mathcal{V}_N^L$.
The approximate best ordering was determined by taking an initial temperature $T_0=1$, a decay rate $\lambda=0.99$ and a maximum number of iterations $i_\mathrm{max}=\frac{1}{2} \binom{L}{N}$. In the left panel, and in the left panels of the subsequent figures, the indicated region within one standard deviation is the region where the logarithms of the singular values lie within one standard deviation of their mean. 

We observe that the Fiedler order improves the distribution of the higher-order singular values of the reshaped tensor beyond, say, the first 150 singular values by an order of magnitude. The best prefactor order and its approximate solution by simulated annealing, by contrast, improve it by four to five orders of magnitude on average.

\begin{figure}[!h]
\centering
\captionsetup[subfigure]{justification=centering}
    \begin{subfigure}[b]{0.48\textwidth}
        \centering
        \includegraphics[width=\textwidth]{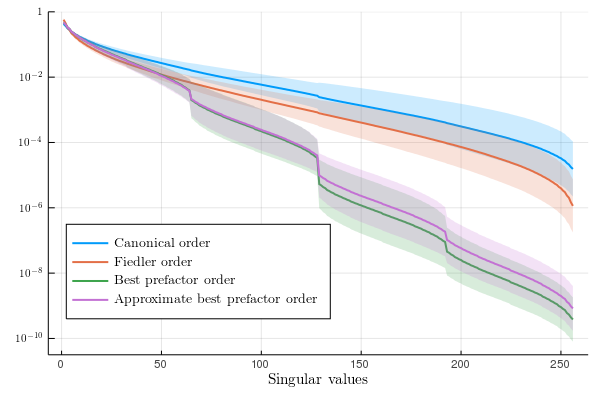}
        \caption{Mean (solid line) and region within one standard deviation (ribbon)}
    \end{subfigure}
        \begin{subfigure}[b]{0.48\textwidth}
            \centering
        \includegraphics[width=\textwidth]{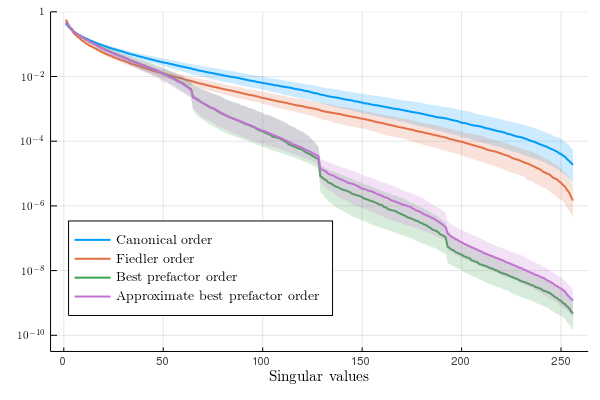}
        \caption{Median (solid line) and 0.25/0.75 quantiles (ribbon)}
    \end{subfigure}
    \caption{Singular values of $\Uppsi^{\mu_1 \dots \mu_{L/2}}_{\mu_{L/2+1}\dots\mu_L}$ for the determinantal state \eqref{eq:tensor} for $N=8$ electrons and $L=16$ orbitals, averaged over 400 simulations.}
\label{fig:slater-comparison}
\end{figure}


\subsection{Tests on correlated states} \label{sec:testcorr}

We have also tested all ordering schemes for sums of several Slater determinants. As a first example, we let 
\begin{equation} \label{weakcorr}
\Psi = \alpha_0 |\psi_1,...,\psi_N\rangle + \alpha_1 |\psi_1,...,\psi_{N-2},\psi_{N+1},\psi_{N+2}\rangle,
\end{equation}
where $\alpha_0=\sqrt{0.9}$, $\alpha_1=\sqrt{0.1}$, and the orbitals $(\psi_k)_{1 \leq k \leq N+2}$ are obtained by generating a partial isometry $U \in \R^{(N+2)\times L}$ as previously.
This state may be called \emph{weakly} correlated since it is a small perturbation of a noninteracting state.

In Figure~\ref{fig:static-comparison}, the distribution of the singular values of $\Uppsi^{\mu_1 \dots \mu_{L/2}}_{\mu_{L/2+1}\dots\mu_L}$ for different ordering schemes are shown. 
The best prefactor order of the dominant Slater is obtained by optimizing the prefactor corresponding to the Slater determinant $|\psi_1,...,\psi_N\rangle$.  
By contrast, the best weighted prefactor order as described in section  \ref{subsec:best_weighted_prefactor} corresponds to simultaneous optimization of the prefactor for both Slater determinants. 
In these numerical tests, the best weighted prefactor order is obtained by an exhaustive search.

\begin{figure}[!h]
\centering
\captionsetup[subfigure]{justification=centering}
    \begin{subfigure}[b]{0.48\textwidth}
        \centering
        \includegraphics[width=\textwidth]{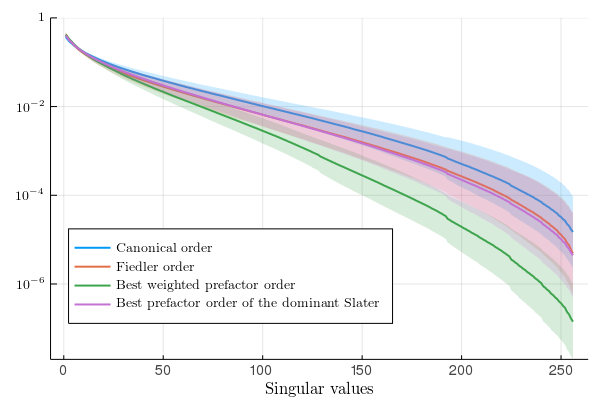}
        \caption{Mean (solid line) and region within one standard deviation (ribbon)}
    \end{subfigure}
        \begin{subfigure}[b]{0.48\textwidth}
            \centering
        \includegraphics[width=\textwidth]{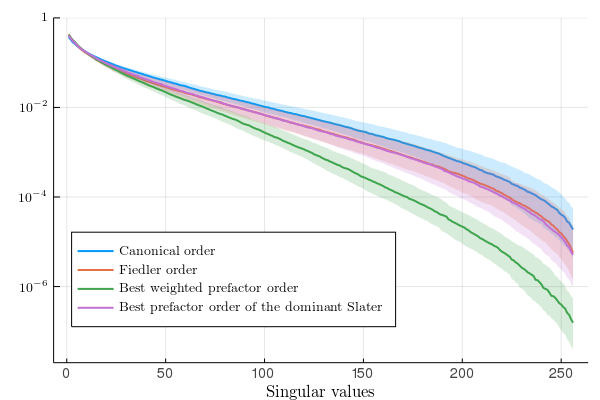}
        \caption{Median (solid line) and 0.25/0.75 quantiles (ribbon)}
    \end{subfigure}
    \caption{Singular values for the weakly correlated state \eqref{weakcorr} for $N=8$ and $L=16$ averaged over 400 simulations}
\label{fig:static-comparison}
\end{figure}

We see that the Fiedler order and the best prefactor order of the dominant Slater give a significant improvement of the canonical order by about an order of magnitude, whereas the best weighted prefactor gives an improvement by about two orders of magnitude.

In Figure~\ref{fig:static-strongly-comparison}, the distribution of the singular values of the following strongly correlated state is plotted:
\begin{equation}\label{corr}
\Psi = \alpha_0 |\psi_1,...,\psi_N\rangle + \alpha_1 |\psi_1,...,\psi_{N-2},\psi_{N+1},\psi_{N+2}\rangle + \alpha_2 |\psi_{N-1},...,\psi_{2N-2}\rangle,
\end{equation}
where $\alpha_0=\sqrt{0.4}$, $\alpha_1=\alpha_2=\sqrt{0.3}$, and the canonical orbitals $(\psi_k)_{1 \leq k \leq 2N-2}$ are obtained by generating a partial isometry $U \in \R^{(2N-2)\times L}$ as previously. 
The considered state is a basic model of a \emph{strongly} correlated state, since all the Slater determiants give roughly the same contribution. We used the same re-shape as before, and again 
averaged the distributions over 400 numerical simulations. 
We observe that when the state is strongly correlated, the canonical, Fiedler or best prefactor orders all yield a comparable decay of the singular values. 
The method we are advocating -- best weighted prefactor -- still improves the decay of the singular values by an order of magnitude, thus outperforming all previous ordering methods. 

\begin{figure}[!h]
\centering
\captionsetup[subfigure]{justification=centering}
    \begin{subfigure}[b]{0.48\textwidth}
        \centering
        \includegraphics[width=\textwidth]{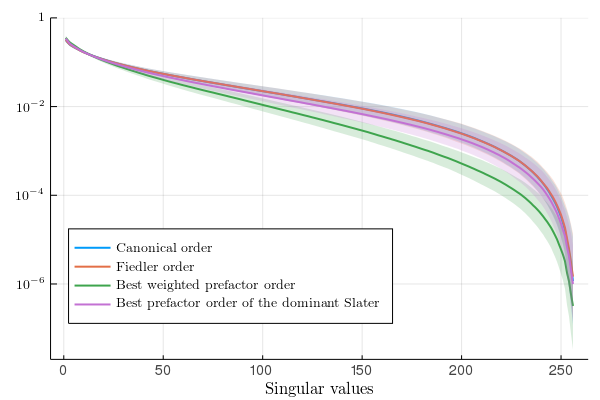}
        \caption{Mean (solid line) and region within one standard deviation (ribbon)}
    \end{subfigure}
        \begin{subfigure}[b]{0.48\textwidth}
            \centering
        \includegraphics[width=\textwidth]{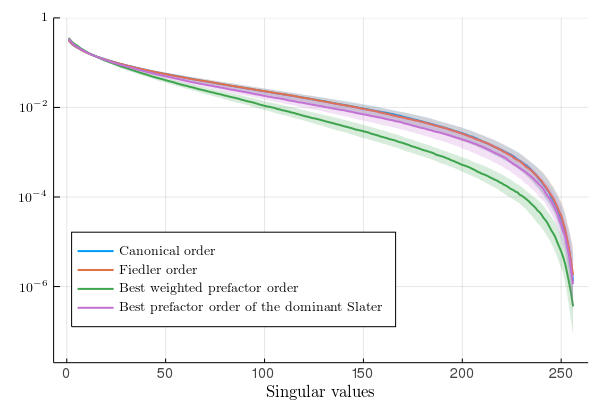}
        \caption{Median (solid line) and 0.25/0.75 quantiles (ribbon)}
    \end{subfigure}
    \caption{Singular values for the strongly correlated state \eqref{corr} for $N=8$ and $L=16$ averaged over 400 simulations}
\label{fig:static-strongly-comparison}
\end{figure}

An important issue beyond the scope of the present work is to investigate the performance of the new ordering scheme within full QC-DMRG simulations of molecular systems. 

\section{Proofs}
\label{sec:proofs}

\subsection{A variant of the Cauchy-Binet formula}

We begin with a small refinement of the Cauchy-Binet formula. This identity will be useful in the following to characterize the singular values of the reshaped tensor of the Slater determinant as well as to compute the one-orbital and two-orbital RDMs. The following proposition can also be found in the Appendix C of \cite{chapman2018classical}.

\begin{prop}
\label{prop:modified_cauchy_binet}
Let $A \in \R^{m \times n}$ and $B \in \R^{n \times m}$. Let $T$ and $U$ be two disjoint subsets of $[n]$ such that $|T|=j$ and $|U|=n-j$. Then 
\begin{equation}
\sum_{S \in \binom{U}{m - j}} \det(A_{[m], S \cup T}) \det(B_{S\cup T, [m]}) = (-1)^{j} \det \begin{pmatrix} \mathbf{0}_{j \times j} & B_{T, [m]} \\
A_{[m], T} & A_{[m], U}B_{U, [m]} \end{pmatrix}  \rm{.}
\end{equation}
\end{prop}

\begin{proof}
The proof of this Proposition relies on the following observation
\begin{equation}
   \begin{pmatrix} \mathbf{0}_{j \times j} & B_{T, [m]} \\
A_{[m], T} & A_{[m], U}B_{U, [m]} \end{pmatrix}  = \begin{pmatrix} \mathrm{Id}_{j\times j} & \mathbf{0}_{j \times j} & \mathbf{0}_{j \times n-j} \\
\mathbf{0}_{m \times j} & A_{[m],T} & A_{[m],U}
\end{pmatrix}
\begin{pmatrix}
\mathbf{0}_{j \times j} & B_{T,[m]} \\
\mathrm{Id}_{j\times j} & \mathbf{0}_{j \times m} \\
\mathbf{0}_{n-j \times j} & B_{U,[m]}
\end{pmatrix} =: \mathcal{A} \mathcal{B}.
\end{equation}
By the Cauchy-Binet formula \cite{horn2013matrix}, we have
\begin{equation}
    \label{eq:calA-calB_cauchy-binet}
    \det\left( \mathcal{A}\mathcal{B} \right) = \sum\limits_{s \in \binom{[n+j]}{m+j}} \det \left( \mathcal{A}_{[m+j],s} \right) \det \left( \mathcal{B}_{s,[m+j]} \right).
\end{equation}
Let $s \in \binom{[n+j]}{m+j}$. If $[j] \not\subset s$, then the columns of $\mathcal{A}_{[m+j],s}$ are linearly dependent, so $\det \left( \mathcal{A}_{[m+j],s} \right)=0$. 
Similarly if $j+[j] \not\subset s$, then the rows of $\mathcal{B}_{s,[m+j]}$ are linearly dependent, so $\det \left( \mathcal{B}_{s,[m+j]} \right)=0$.
Hence, for nonzero terms in \eqref{eq:calA-calB_cauchy-binet}, $s$ is of the form $s = [2j] \cup \tau $ where $\tau \in 2j+\binom{[n-j]}{m-j}$.
For such an $(m+j)$-combination $s$, there is some $S \in \binom{U}{m-j}$ such that
\begin{equation}
    \det \left( \mathcal{A}_{[m+j],s} \right) = \det( A_{[m],S \cup T} ), \quad \text{and} \quad \det \left( \mathcal{B}_{s,[m+j]} \right) = (-1)^j \det( B_{S \cup T,[m]} ),
\end{equation}
which concludes the proof.
\end{proof}

Notice that for $j=0$, we recover the usual Cauchy-Binet formula.

\subsection{Proof of Theorem~\ref{thm:sing_values_TT_of_slater}}

Since $\Psi$ is a wave function of an $N$-electron Slater determinant, without loss of generality the matrix $\Psi^{\mu_1,\dots,\mu_{k}}_{\mu_{k+1},\dots,\mu_L}$ can be reordered up to permutations of its columns and rows  in a block diagonal form
\begin{equation}
    \label{eq:full_tensor_reshape}
    \Psi^{\mu_1,\dots,\mu_{k}}_{\mu_{k+1},\dots,\mu_L}  = \begin{blockarray}{ccccc}
    & \underbrace{0 \cdots 1}_{N \text{ occurences of } 1} & \underbrace{0 \cdots 1}_{N-1 \text{ occurences of } 1} & \cdots & \underbrace{0\cdots0}_{0 \text{ occurence of } 1} \\
    \begin{block}{c(cccc)}
    \underbrace{0 \cdots 0}_{0 \text{ occurence of } 1} & C_0 & 0 & \cdots & 0 \\
    \underbrace{0 \cdots 1}_{1 \text{ occurence of } 1} & 0 & C_1 & \dots & 0 \\
    \vdots & \vdots & & \ddots & \vdots \\
    \underbrace{0 \cdots 1}_{N \text{ occurences of } 1} & 0 & \cdots & \cdots & C_N \\
    \end{block}
\end{blockarray}.
\end{equation}
The submatrices $C_j$ for ${0 \leq j \leq N}$ have dimensions $\binom{k}{j} \times \binom{L-k}{N-j}$.

%

\begin{lem}
\label{lem:2e_formule}
Let $C_{j}$ be the matrix defined in Equation~\eqref{eq:full_tensor_reshape}. If $k \leq L-N$, for all $\sigma,\tau \in \binom{[k]}{j}$, we have
\[
(C_{j} C_{j}^T)_{\sigma,\tau} = \det(W_kW_k^T) \det \left[ \left(V_k^T (W_kW_k^T)^{-1}V_k \right)^\tau_\sigma \right],
\]
where $\left(V_k^T (W_kW_k^T)^{-1}V_k \right)^\tau_\sigma$ is the submatrix of rows $\lbrace \sigma(1), \dots, \sigma(j) \rbrace$ and columns $\lbrace \tau(1), \dots, \tau(j) \rbrace$.

This expression can be written in a matrix form
\[
C_{j} C_{j}^T = \det (W_k W_k^T)\, \Lambda^j\left(V_k^T (W_kW_k^T)^{-1}V_k \right),
\]
where $\Lambda^j A$ denotes the $j$-th compound matrix of a matrix $A$, \emph{i.e.} the matrix of minors of order $j$ of $A$ in lexicographical order.

Similarly, if $k \geq N$, for all $\sigma,\tau \in \binom{[L-k]}{N-j}$ we have
\[
(C_{j}^T C_{j})_{\sigma,\tau} = \det(V_kV_k^T) \det \left[ \left(W_k^T (V_kV_k^T)^{-1}W_k \right)^\tau_\sigma \right],
\]
so that in matrix notation
\[
C_{j}^T C_{j} = \det (V_k V_k^T)\, \Lambda^{N-j}\left(W_k^T (V_kV_k^T)^{-1}W_k \right).
\]
\end{lem}

The proof of this lemma relies on Proposition~\ref{prop:modified_cauchy_binet}.

\begin{proof}[Proof of Lemma~\ref{lem:2e_formule}]
Let $\sigma,\tau \in \binom{[k]}{j}$. By the definition of $C_{j}$, the coefficient $(C_{j}C_{j}^T)_{\sigma,\tau}$ is given by
\[
(C_{j}C_{j}^T)_{\sigma,\tau} = \sumlim{\rho \in k+\binom{[L-k]}{N-j}}{} \det( u_\sigma u_\rho ) \det( u_\tau u_\rho ).
\]
Using Proposition~\ref{prop:modified_cauchy_binet} and Assumption~\ref{assumption:invertibility}, we have
\begin{align*}
 (C_{j}C_{j}^T)_{\sigma,\tau} & = (-1)^{j} \det \begin{pmatrix}
    \mathbf{0}_{j \times j}  &  u_\tau^T  \\
      u_\sigma & W_kW_k^T \\
 \end{pmatrix} \\  
 & = (-1)^{j} \det \left(W_kW_k^T \right) \det \left( -u_\tau^T(W_kW_k^T)^{-1}u_\sigma \right) \\
 & = \det \left(W_kW_k^T \right) \det \left[ \left( V_k^T(W_kW_k^T)^{-1}V_k \right)^\tau_\sigma \right].
\end{align*}

The second part of the lemma is proved the same way.
\end{proof}

We now have all the ingredients to prove Theorem~\ref{thm:sing_values_TT_of_slater}.

\begin{proof}[Proof of Theorem~\ref{thm:sing_values_TT_of_slater}]
We prove a result slightly stronger than the statement of Theorem~\ref{thm:sing_values_TT_of_slater}. Namely, we will show that the positive singular values $(\alpha)_i$ of $C_j$ and $(\beta)_i$ of $C_{\min(k,N,L-k)}$ satisfy 
\[
    \forall \, 1 \leq j \leq \min \big( \tbinom{k}{j}, \tbinom{L-k}{N-j} \big) , \ {\alpha_i}^2 {\beta_{d-j}}^2 = 
    \begin{cases} 
    \det(V_k^T V_k) \det(W_kW_k^T), & \text{ if } k \leq N \\
    \det(V_k V_k^T) \det(W_kW_k^T), & \text{ if } N \leq k \leq L-k \\
    \det(V_k V_k^T) \det(W_k^TW_k), & \text{ if } L-k \leq k \leq L-1 
    \end{cases}
\]

The proof of Theorem~\ref{thm:sing_values_TT_of_slater} is divided into three cases.

\paragraph{\textbf{Case 1} $k \leq N$:}

By Lemma~\ref{lem:2e_formule}, we know for $0 \leq j \leq k$ that
\begin{align*}
    C_{k-j} C_{k-j}^T & = \det (W_k W_k^T)\, \Lambda^{k-j}\left(V_k^T (W_kW_k^T)^{-1}V_k \right), \\
    C_{j} C_{j}^T & = \det (W_k W_k^T)\, \Lambda^{j}\left(V_k^T (W_kW_k^T)^{-1}V_k \right).
\end{align*}
The eigenvalues of $\Lambda^{j}\left(V_k^T (W_kW_k^T)^{-1}V_k \right)$ are the products of $j$ distinct eigenvalues (counted with their multiplicities) of the matrix  $\Lambda^{j}\left(V_k^T (W_kW_k^T)^{-1}V_k \right)$ \cite[Theorem~6.18]{fiedler1986special}.
Let $V_k=A\Sigma B^T$ be the singular value decomposition of $V_k$ where $A \in \R^{N\times N}$ and $B \in \R^{k \times k}$ are orthogonal matrices and $\Sigma \in \R^{N \times k}$ is a diagonal matrix with the singular values of $V_k$. 
By Assumption~\ref{assumption:invertibility}, $V_k$ and $W_k$ are full rank matrices. 
Hence the singular values $(s_i)_{1 \leq i \leq k}$ of $V_k$ satisfy $0 < s_i <1$.
Thus $V_k^T(W_kW_k^T)^{-1}V_k = B^T\Sigma (I_N - \Sigma\Sigma^T)^{-1} \Sigma B$. 
Since $\binom{[k]}{j}$ and $\binom{[k]}{k-j}$ are in one-to-one correspondence, to an eigenvalue $\lambda$ of $\Lambda^{j}\left(V_k^T (W_kW_k^T)^{-1}V_k \right)$ there corresponds exactly one eigenvalue $\mu$ of $\Lambda^{k-j}\left(V_k^T (W_kW_k^T)^{-1}V_k \right)$ such that $\lambda \mu = \det(\Sigma (I_N - \Sigma\Sigma^T)^{-1} \Sigma^T) = \det(V_k^TV_k) \det(W_kW_k^T)^{-1}$. The result now follows.

\paragraph{\textbf{Case 2} $k \geq L-N$:}

We repeat the proof by considering $C_{j}^T C_{j}$ and $C_{k-j}^T C_{k-j}$ instead of $C_{j} C_{j}^T$ and $C_{k-j} C_{k-j}^T$.

\paragraph{\textbf{Case 3} $N \leq k \leq L-N$:}

By Lemma~\ref{lem:2e_formule}, we have for $0 \leq j \leq N$ 
\begin{align*}
    C_{N-j} C_{N-j}^T & = \det (W_k W_k^T)\, \Lambda^{N-j}\left(V_k^T (W_kW_k^T)^{-1}V_k \right), \\
    C_{j}^T C_{j} & = \det (V_k V_k^T)\, \Lambda^{N-j}\left(W_k^T (V_kV_k^T)^{-1}W_k \right).
\end{align*}
Using the singular values $(s_i)_{1 \leq i \leq N}$ of $V_k$, the nonzero eigenvalues of $V_k^T (W_kW_k^T)^{-1}V_k$ are $(\frac{s_i^2}{1-s_i^2})_{1 \leq i \leq N}$ and those of $W_k^T (V_kV_k^T)^{-1}W_k$ are $(\frac{1-s_i^2}{s_i^2})_{1 \leq i \leq N}$. By a bijection argument, a nonzero eigenvalue of $V_k^T (W_kW_k^T)^{-1}V_k$ corresponds exactly to the inverse of one eigenvalue of $W_k^T (V_kV_k^T)^{-1}W_k$. 
\end{proof}

\subsection{Proof of Theorem~\ref{cor:strongly_correlated_hosv}}

The proof of Theorem~\ref{cor:strongly_correlated_hosv} is a consequence of the following generalized Weyl inequality on the eigenvalues of the sum of symmetric matrices.

\begin{prop}[Generalized Weyl inequality]
Let $B = \sum\limits_{k=1}^L A_k$ where $A_k \in \mathbb{R}^{n \times n}$ is symmetric. Let $(\lambda_i^{(B)})_{1 \leq i \leq n}$ and $(\lambda_i^{(A_k)})_{1 \leq i \leq n}$ be respectively the eigenvalues in decreasing order of the matrices $B$ and $A_k$, $k=1,\dots,L$. Then we have for all $i,i_1,\dots i_{L-1}$ such that $0 \leq i_k \leq n-1$ and $\sum\limits_{k=1}^{L-1} i_k \leq i+1$
\begin{equation}
    \lambda_{i}^{(B)} \leq \lambda_{i_1+1}^{(A_1)} + \cdots + \lambda_{i_{L-1}+1}^{(A_{L-1})} + \lambda_{i-\scriptscriptstyle{\sum\limits_{\scriptscriptstyle{k=1}}^{\scriptscriptstyle{L-1}}i_{\scriptscriptstyle{k}}}}^{(A_L)}.
\end{equation}
\end{prop}

\begin{proof}
The proof relies on the following lemma \cite[Lemma 4.2.3]{horn2013matrix}: if $S_1,\dots,S_k$ are subspaces of $\mathbb{R}^n$ such that $\sum\limits_{j=1}^k \dim(S_j) \geq (k-1)n+1$, then $S_1 \cap \cdots \cap S_k$ contains a unit vector.

Let $(x_j^{(B)})_{1 \leq j \leq n}$ and $(x_j^{(A_k)})_{1 \leq j \leq n}$ be respectively eigenvectors associated to $(\lambda_j^{(B)})_{1 \leq j \leq n}$ and $(\lambda_j^{(A_k)})_{1 \leq j \leq n}$. Define 
\begin{align*}
    S^{(B)} & = \mathrm{Span}\left( x_1^{(B)}, \dots, x_i^{(B)} \right), \\
    S^{(A_k)} & = \mathrm{Span}\left( x_{i_k+1}^{(A_k)},\dots,x_n^{(A_k)} \right), \text{ for } k=1,\dots,L-1 \\
    S^{(A_L)} & = \mathrm{Span}\left( x_{i - \sum\limits_{k=1}^{L-1}i_k}^{(A_L)},\dots,x_n^{(A_L)} \right).
\end{align*}
Then by construction, $\dim \left(S^{(B)}\right) + \sum\limits_{k=1}^{L} \left(S^{(A_k)}\right) \geq Ld+1$. Hence the intersection of these subspaces has a unit vector. Denote this unit vector by $x_*$. Since the eigenvalues are in decreasing order, we have
\begin{equation}
    \lambda_i^{(B)} \leq x_*^T B x_* \leq \lambda_{i_1+1}^{(A_1)} + \cdots + \lambda_{i_{L-1}+1}^{(A_{L-1})} + \lambda_{i-\sum\limits_{k=1}^{L-1}i_k}^{(A_L)}.
\end{equation}
\end{proof}

\begin{proof}[Proof of Theorem~\ref{cor:strongly_correlated_hosv}]
For a given matrix $A \in \mathbb{R}^{m \times n}$, the singular values of $A$ are the eigenvalues of the matrix $\begin{pmatrix} 0 & A \\ A^T & 0 \end{pmatrix}$.
Hence the Weyl inequality on the eigenvalues of the sum of symmetric matrices can be extended to singular values of sums of matrices. 
\end{proof}

\subsection{Proof of Proposition~\ref{prop:orbital_density_matrix}}
\label{subsec:proof-prop-RDM-Slater}

\begin{proof}[Proof of Proposition~\ref{prop:orbital_density_matrix}]
We begin with the one-orbital density matrix. Note that the off-diagonal terms of the one-orbital density matrix are equal to 0 because the Slater determinant is an $N$-body state. 
By definition, the unoccupied-unoccupied entry of the one-orbital RDM is
\[
\rho_i^{(1)}(0,0) = \sum\limits_{\substack{\mu_k \in \lbrace 0,1 \rbrace \\ \mu_i = 0}} |C_{\mu_1,\dots,\mu_L}|^2 = \sum\limits_{\substack{1 \leq i_1 < \cdots < i_N \leq L \\ i_k \not= i }} |\det(u_{i_1} \cdots u_{i_N})|^2,
\]
where we used formula \eqref{slater_coefficient}. Using the Cauchy-Binet formula for the matrix $U_{-i} = (u_1 \cdots u_{i-1} u_{i+1} \cdots u_L)$, we have
\[
\rho_i^{(1)}(0,0) = \det(U_{-i}U_{-i}^T) = \det(\mathrm{Id}_N - u_{i}u_{i}^T) = 1-\|u_i\|^2.
\]
Using that the Slater determinant $\Psi$ is normalized, we have $\rho^{(1)}_i(1,1) = 1-\rho^{(1)}_i(0,0) = \|u_i\|^2$.

For the two-orbital RDM, the computations are similar. We will give the details for the first diagonal and the off-diagonal entries. 
We have
\begin{equation*}
    \rho^{(2)}_{i,j}(00,00) = \sum\limits_{\substack{\mu_k \in \lbrace 0,1 \rbrace \\ \mu_i = \mu_j = 0}} |C_{\mu_1,\dots,\mu_L}|^2 = \sum\limits_{\substack{1 \leq i_1 < \cdots < i_N \leq L \\ i_k \not= i,j }} |\det(u_{i_1} \cdots u_{i_N})|^2.
\end{equation*}
Using the Cauchy-Binet formula for the matrix 
$U_{-i,-j} = \left( u_1 \cdots u_{i-1} u_{i+1} \cdots u_{j-1} u_{j+1} \cdots u_L \right)$, 
we get
\begin{align*}
   \rho^{(2)}_{i,j}(00,00) &= \det \left( \mathrm{Id}_N - (u_i u_j)\begin{pmatrix}
   u_i^T  \\ u_j^T
   \end{pmatrix} \right) \\
   &= \det \begin{pmatrix}
   1 - \|u_i\|^2 & u_i^T u_j \\
   u_j^Tu_i & 1 - \|u_j\|^2
   \end{pmatrix}
   = 1- \|u_i\|^2- \|u_j\|^2+G,
\end{align*}
where $G = \|u_i\|^2 \|u_j\|^2 -(u_i^Tu_j)^2$.

The off-diagonal term $\rho^{(2)}_{i,j}(10,01)$ is given by
\begin{align*}
\rho^{(2)}_{i,j}(10,01) &= \sum\limits_{\mu_k \in \lbrace 0,1 \rbrace} C_{\mu_1,\dots,\mu_{i-1},1,\mu_{i+1},\dots,\mu_{j-1},0,\mu_{j+1},\dots,\mu_L} C_{\mu_1,\dots,\mu_{i-1},0,\mu_{i+1},\dots,\mu_{j-1},1,\mu_{j+1},\dots,\mu_L} \\
   &= \sum\limits_{\substack{1 \leq i_1 < \cdots < i_N \leq L \\ i_k \not= i,j }} \det(u_{i_1} \cdots u_i \cdots u_{i_N}) \det(u_{i_1} \cdots u_j \cdots u_{i_N}) \\
   &= \sum\limits_{\alpha \in \binom{[L]\setminus \lbrace i,j \rbrace}{N-1}} \det(U_{\alpha \cup i}) \det(U_{\alpha \cup j}).
\end{align*}
We want to apply Proposition~\ref{prop:modified_cauchy_binet}. Hence we partition the $(N-1)$-combination $\alpha$ into $\alpha = \gamma \cup \beta$, where $\gamma \in \binom{[i+1:j-1]}{k}$ and $\beta \in \binom{[i-1] \cup [j+1:L]}{N-1-k}$.
Hence we obtain
\begin{align*}
\rho^{(2)}_{i,j}(10,01) &= \sum\limits_{k=0}^{j-i-1} \sum\limits_{\gamma \in \binom{[i+1:j-1]}{k}} \sum\limits_{\beta \in \binom{[i-1] \cup [j+1:L]}{N-1-k}} \det(U_{\beta \cup i \cup \gamma}) \det(U_{\beta \cup \gamma \cup j}) 
\end{align*}
For $0 \leq k \leq j-i-1$ and $\gamma \in \binom{[i+1:j-1]}{k}$, by Proposition~\ref{prop:modified_cauchy_binet} we have
\[
\sum\limits_{\beta \in \binom{[i-1] \cup [j+1:L]}{N-1-k}} \det(U_{\beta \cup i \cup \gamma}) \det(U_{\beta \cup \gamma \cup j}) = (-1)^{k+1} \det \begin{pmatrix} \mathbf{0}_{k+1 \times k+1} & U_{\gamma \cup j}^T \\
U_{i \cup \gamma} & U_{-[i,j]} U_{-[i,j]}^T  \end{pmatrix},
\]
where $U_{-[i,j]} = (u_1 \cdots u_{i-1} u_{j+1} \cdots u_L)$. Since $UU^T = \mathrm{Id}_N$, we have
\[
U_{-[i,j]} U_{-[i,j]}^T = \mathrm{Id}_N - U_{[i,j]}U_{[i,j]}^T.
\]
Using the alternating property of the determinant we obtain the asserted result. 
\end{proof}
\vspace*{-3mm}

\begin{small}

\end{small}

\end{document}